\documentclass{amsart}
\usepackage{graphicx} 
\usepackage{bm}
\usepackage{amssymb}
\usepackage{enumerate}
\usepackage{mathtools}
\usepackage{esint}
\usepackage{cite}

\newenvironment{nouppercase}{%
  \renewcommand{\uppercasenonmath}[1]{}}{}
\title[\Large Localization Coefficients of Functions]{ \Large{Localization Coefficients of Functions with \\Applications in Partial Differential Equations}  }

\author{Mirza Karamehmedovi\'c$^\star$}
\thanks{$\star$ Department of Applied Mathematics and Computer Science, Technical University of Denmark, Kgs. Lyngby, Denmark (mika@dtu.dk).}

\author{Faouzi Triki$^\ddag$}

\thanks{$\ddag$ Laboratoire Jean Kuntzmann, Université Grenoble-Alpes, Grenoble, France (Faouzi.Triki@univ-grenoble-alpes.fr).}

\date{}

\newcommand{\R}{\bm{R}}
\newcommand{\N}{\bm{N}}

\newcommand{\Hdot}{\dot H}

\newtheorem{lemma}{Lemma}

\newtheorem{definition}{Definition}
\newtheorem{proposition}{Proposition}
\newtheorem{remark}{Remark}

\begin{document}
\begin{nouppercase}
\maketitle
\end{nouppercase}

\begin{abstract}
We identify shortcomings in two popular measures of localization of functions: the \(L^p\text{-}L^q\) participation ratio and the mass concentration comparison. We then introduce a novel localization measure for functions on bounded subsets of \(\bm{R}^d\), $d=1,2,3,\dots$, based on a Wasserstein metric. For efficient computation, we prove the equality of this measure with a suitable Sobolev norm in dimension one. We demonstrate our approach by numerical experiments in one and two dimensions. Finally, we discuss and mitigate challenges arising from boundary effects.

\end{abstract}

\section{Introduction}


Localization is a phenomenon in solutions of differential equations across quantum mechanics, electrodynamics, acoustics, and other fields, where certain eigenfunctions become concentrated in small regions of the domain while remaining nearly zero elsewhere.~\cite{Filoche-2012-2,Filoche-2012,Heilman-2010,Yamilov-2023,Felix-2007,burq2024delocalized,Vazquez-1982,2024-localization_1D,ammari2024anderson} This phenomenon often turns out to be of technological as well as scientific interest. In addition to the well-known example of Anderson localization of electromagnetic waves in disordered media~\cite{Yamilov-2023}, photonic nanojets~\cite{Darafsheh-2021,2022-PNJ1,2023-phase-only_PNJ} likewise exemplify a technologically promising phenomenon rooted in wave localization. Finally, one can certainly also study the localization of functions without an inherent connection to a differential equation.

Despite the extensive interest in localization, a rigorously justified, robust, and universal method for measuring it remains elusive. Consequently, characterizing localization-driven phenomena—such as defining the size of a photonic nanojet—often involves some degree of arbitrariness. Localization can be measured using mass concentration comparison~\cite{hardy1929some,kesavan2006symmetrization}. Here, the pointwise information on the given functions is compounded into monotonically decaying 'rearrangements' whose local decay properties provide insight into the degree of localization. Next, a substantial body of literature~\cite{Filoche-2012,Grebenkov-2013} quantifies the localization of functions in $L^p(\Omega)\cap L^q(\Omega)$, $\Omega\subseteq\R^d$, $d\in\N$, in terms of their participation ratios
\[
\alpha_{p,q}(u,\Omega)=|\Omega|^{1/q-1/p}\frac{\|u\|_p}{\|u\|_q},\quad 1\le p<q\le\infty.
\]
Specifically, the localization lemma~\cite[Lemma 3]{vandenBerg-2021} establishes the equivalence between the localization of a sequence $(u_n)_{n\in\N}$ of functions on a sequence $(\Omega_n)_{n\in\N}$ of sets, in the sense of~\cite{vandenBerg-2021-2}, and the property \[
\lim_{n\rightarrow\infty}\alpha_{1,2}(u_n,\Omega_n)=0.
\]
For a fixed domain $\Omega$, the parameter $\alpha_{2,4}(u,\Omega)$ is known empirically~\cite{2024-localization_1D} to attain relatively small values when $u$ is highly localized in $\Omega$, and indeed $\alpha_{2,4}(u,\Omega)$ is a commonly used measure of localization~\cite{Filoche-2012,Felix-2007}. In~\cite{2024-localization_1D} we provided asymptotic and non-asymptotic bounds on $\alpha_{2,4}(\phi_n,I)$ for eigenfunctions $\phi_n$, $n\in\N$, of regular Sturm-Liouville operators on an interval $I\subset\R$, confirming several expected connections between the value of that localization coefficient and the eigenfunction index $n$. However, any participation ratio $\alpha_{p,q}(u,\Omega)$ involves only the Lebesgue norms of the function $u$, and in Section~\ref{sec:mass_Lebesgue} we use the concept of rearrangement to illustrate the loss of information involved in such strictly Lebesgue norm-based localization coefficients, including the mass concentration comparison approach. Finally, to our knowledge, the use of $\alpha_{2,4}(u,\Omega)$ has so far not been motivated formally in the literature in the general case.

Another common way to measure the concentration of a probability density $u$ on $\Omega$ is in terms of its variance
\[
\sigma^2(f)=\int_{\Omega}|x-x_{{\rm cm}(u)}|^2u(x)dx,
\]
with $x_{{\rm cm}(u)}$ the center of mass of $u$,
\[
x_{{\rm cm}(u)}=\frac{1}{|\Omega|}\int_{\Omega}xu(x)dx.
\]
The variance is in fact the square of the Wasserstein-2 distance of the multiplicative measure $\mu_u(x)=u(x)dx$ on $\Omega$ from the maximally localized, degenerate probability density $\delta_{{\rm cm}(u)}$ on $\Omega$, which in turn is an intuitive, optimal mass transport-related measure of localization. However, the involvement of the Dirac delta makes this approach less favourable when maximizing the localization within sets of non-degenerate probability densities. It also makes it incompatible with an equivalent PDE-based interpretation of the Wasserstein-2 distance that offers substantial numerical advantages, as explained in detail below.

Our purpose here is
\begin{enumerate}
    \item to demonstrate a deficiency of measures of localization based on the $L^p-L^q$ participation ratio, or on the comparison of mass concentrations;
    \item to provide an intuitively and formally justified, robust, computationally efficient definition of localization, by taking as the starting point the Wasserstein-2 distance between nondegenerate probability densities, and
    \item to support the above developments using relevant numerical examples. 
\end{enumerate} 
In Section~\ref{sec:mass_Lebesgue} we discuss comparisons of mass concentrations as a means to measure localization, and we use the notions developed there to demonstrate a drawback of Lebesgue norm-based localization coefficients. Then, in Section~\ref{sec:new} we introduce our new localization coefficient $\beta(u,\Omega)$, which is based on the Wasserstein-2 distance between the probability distribution $u$ on $\Omega$ and the least localized probability distribution $|\Omega|^{-1}$ on $\Omega$. This leads naturally to the definition of localization in terms of a homogeneous Sobolev space distance between $u$ and $|\Omega|^{-1}$. In Section~\ref{sec:new}, we also show and address a boundary effect distorting localization measurements and occurring when some of the 'mass' of a probability density on $\Omega$ is close to the boundary of $\Omega$. This effect has deep consequences for the general definition of localization. In Section~\ref{sec:numerical} we showcase the performance of $\beta(u,\Omega)$ using numerical examples in dimension one and two, highlighting the advantages of our localization coefficient.

\section{Mass concentration comparisons and\\Lebesgue norm-based localization coefficients}\label{sec:mass_Lebesgue}
In this section we measure localization in terms of mass concentration comparisons. Assume $d\in\N=\{1,2,\dots\}$, $\Omega\subseteq\R^d$ is an open set, $|\cdot|$ is the Lebesgue measure on $\R^d$, and all considered functions $u,v$ are real-valued. Recall from, e.g.,~\cite[pp. 1-5]{kesavan2006symmetrization}, that the distribution function $\mu_u$ of $u$ is given by
\[
\mu_u(k)=|\{x\in\Omega,\,\,|u(x)|>k\}|,\quad k\ge0,
\]
the decreasing rearrangement of $u$ is given by
\[
u^{\ast}(s)=\sup\{k\ge0,\,\,\mu_u(k)>s\},\quad s\in(0,|\Omega|),
\]
and the spherical decreasing rearrangement of $u$ is given by
\[
u^{\#}(x)=u^{\ast}(|B_1(0)||x|^n),\quad x\in\Omega^{\#},
\]
where $B_1(0)$ is the unit ball in $\R^d$ centered at the origin, and $\Omega^{\#}$ is the ball in $\R^d$ centered at the origin and satisfying $|\Omega^{\#}|=|\Omega|$.
The following are now well-known facts~\cite[pp. 1-5]{kesavan2006symmetrization}:
\begin{enumerate}[(i)]
\item $\|u\|_{L^p(\Omega)}=\|u^{\ast}\|_{L^p(0,|\Omega|)}=\|u^{\#}\|_{L^p(\Omega^{\#})}$ for all $p\in[1,\infty]$,
\item $u^{\ast}$ is the generalized inverse of $\mu_u$, and
\item $\mu_u=\mu_{u^{\ast}}$.
\end{enumerate}
\begin{definition}[Comparison of mass concentrations]\label{def:1}
Let $u, v\in L^{\infty}(\Omega)\setminus\{0\}$. We say that $u$ is less concentrated than $v$, and we write $u \prec v$, if
\begin{equation}\label{eqn:mass}
\int_0^{|\Omega|}(v^{\bigstar}-u^{\bigstar})_+dr\le\int_0^{|\Omega|}(u^{\bigstar}-v^{\bigstar})_+dr,
\end{equation}
where $u^{\bigstar}=u^{\ast}/\|u\|_{L^{\infty}(\Omega)}$, $v^{\bigstar}=v^{\ast}/\|v\|_{L^{\infty}(\Omega)}$, and $h_+=\max\{h,0\}$ for any real-valued $h\in L^{\infty}([0,|\Omega|])$.
\end{definition}
The relation $\prec$ defines a total ordering on $L^{\infty}(\Omega)\setminus\{0\}$ up to some equivalence class. Indeed, given $u,v\in L^{\infty}(\Omega)\setminus\{0\}$, the functions $u^{\bigstar},v^{\bigstar}$ are nonnegative-valued, and for any real-valued $h\in L^{\infty}([0,|\Omega|])$ we have
\[
h_+(r)-(-h)_+(r)=\max\{h(r),0\}-\max\{-h(r),0\}=h(r),\quad r\in[0,|\Omega|],
\]
so the inequality~\eqref{eqn:mass} is equivalent with
\[
\|v^{\bigstar}\|_{L^1([0,|\Omega|])}\le\|u^{\bigstar}\|_{L^1([0,|\Omega|])}.
\]
Thus, for any $u,v,h\in L^{\infty}(\Omega)\setminus\{0\}$ we have 
\begin{enumerate}
    \item $u\prec v$ or $u\prec v$,
    \item $u\prec u$,
    \item $u\prec v$ and $v\prec h$ implies $u\prec h$.
\end{enumerate}
Finally, we can have pointwise different $u,v\in L^{\infty}(\Omega)\setminus\{0\}$ satisfying $u\prec v$ and $v\prec u$.

We can in fact simplify our definition of the ordering $\prec$:
\begin{proposition}
    Let $u, \, v \in  L^\infty(\Omega)\setminus\{0\}$. Then $u \prec v$ if and only if $\alpha_{1,\infty}(v,\Omega) \leq \alpha_{1,\infty}(u,\Omega)$. 
\end{proposition}
\begin{proof}
Since  
\[
\|u\|_{L^p(\Omega)} = \|u^*\|_{L^p(0, |\Omega|)}\quad\text{and}\quad \|v\|_{L^p(\Omega)} = \|v^*\|_{L^p(0, |\Omega|)},\quad p \in 
\N^* \cup \{\infty\},
\]
we have 
$\alpha_{1,\infty}(v,\Omega)= \alpha_{1,\infty}(v^\bigstar,[0,|\Omega|])$ and $\alpha_{1,\infty}(f,\Omega) = \alpha_{1,\infty}(f^\bigstar,[0,|\Omega|])$. Assume, without loss of generality, that $ \|u\|_{L^\infty(\Omega)} = \|v\|_{L^\infty(\Omega)}$. Now $\alpha_{1,\infty}(v,\Omega) \leq \alpha_{1,\infty}(u,\Omega)$ is equivalent with
\[
\|v\|_{L^{\infty}(\Omega)}\|v^{\bigstar}\|_{L^1([0,|\Omega|])}=\|v^*\|_{L^1([0,|\Omega|])}\leq\|u^*\|_{L^1([0,|\Omega|])}=\|u\|_{L^{\infty}(\Omega)}\|u^{\bigstar}\|_{L^1([0,|\Omega|])},
\]
that is, with
\begin{equation}\label{eqn:p1}
\|v^{\bigstar}\|_{L^1([0,|\Omega|])}\le\|u^{\bigstar}\|_{L^1([0,|\Omega|])},
\end{equation}
and we have already shown the equivalence of~\eqref{eqn:p1} with $u\prec v$.
\end{proof}

V\'asquez and Volzone~\cite[Definition 2.1]{Vazquez-2014} define $u$ as less concentrated than $v$ if
\[
\int_{B_R(0)}u^{\#}(x)dx\le\int_{B_R(0)}v^{\#}(x)dx
\]
for all $R\in(0,|\Omega|)$, where $B_R(0)$ is the ball in $\R^n$ centered at the origin and of radius $R$. This is the so-called comparison of mass concentrations, and our Definition~\ref{def:1} can be seen as a special case where $R=|\Omega|$. However, the comparison of mass concentrations involves only Lebesgue norm information about $u$ and $v$, and is therefore susceptible to the same kind of issues as now discussed for $\alpha_{p,q}(u,\Omega)$, or any strictly $L^p$ norm-based measure of localization:
\begin{proposition}\label{prop:lp}
If $\mu_{u_1}\equiv\mu_{u_2}$ then $\|u_1\|_p=\|u_2\|_p$ for all $p\in[1,\infty]$.
\end{proposition}
\begin{proof}
Since $u_j^{\ast}$ is the generalized inverse of $\mu_{u_j}$, we have $u_1^{\ast}=u_2^{\ast}$, and the result now follows from $\|u_j\|_p=\|u^{\ast}_j\|_p$.
\end{proof}
As a simple example, we consider a family
\begin{equation}\label{eqn:uad}
u_{a,d}=\widetilde u_{a,d}/\|\widetilde u_{a,d}\|_{L^1(0,1)}
\end{equation}
of superpositions of two step functions, with
\begin{equation}\label{eqn:j1}
\widetilde u_{a,d}(x)=\varphi(x-a)+\varphi(x-a-d),\quad x\in(0,1),
\end{equation}
and
\begin{equation}\label{eqn:j2}
\varphi(x)=\begin{cases}1,\quad&|x|\le b,\\0,\quad&|x|>b,\end{cases}
\end{equation}
as illustrated in Figure~\ref{fig:alpha_fail} (left). Clearly, for sufficiently small $b$ and sufficiently large $d_1,d_2$, we have $\mu_{u_{a,d_1}}\equiv\mu_{u_{a,d_2}}$, and the localization coefficient $\alpha_{2,4}(u,\Omega)$ is therefore unable to discern $u_{a,d_1}$ from $u_{a,d_2}$ in such cases, although $u_{a,d_1}$ is more localized than $u_{a,d_2}$ when $0\le d_1<d_2$. This example also illustrates that \textit{any} measurement of localization based on $\mu_u$ alone generally fails.

\begin{figure}[hbt!]
    \centering
    \includegraphics[width=0.49\linewidth]{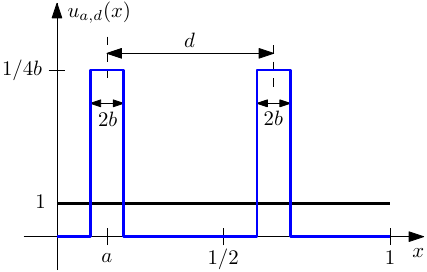}
    \caption{The distribution function $\mu_{u_{a,d}}$ of $u_{a,d}$ from~\eqref{eqn:uad}--\eqref{eqn:j2} is independent of $d\in[2b,1-a-b]$, and so is any $L^p$ norm of $u_{a,d}$. Thus mass concentration comparisons involving $\mu_{u_{a,d}}$ and the constant-valued function $1$, as well as all participation ratios $\alpha_{p,q}(u_{a,d},(0,1))$, are invariant to changes in $d$, in spite of the fact that $u_{a,d}$ delocalizes as $d$ grows.}
    \label{fig:alpha_fail}
\end{figure}
\begin{remark}
  Two real valued functions (with possibly different domains of definition) that have the same distribution function  are said to be equimeasurable \cite[Definition 1.1.2]{kesavan2006symmetrization}.
  Actually the equivalence relation $u_1\sim u_2$ iff $\mu_{u_1} \equiv \mu_{u_2}$  provides a partition of $L^1(\Omega)$ into disjoint equivalence classes of equimeasurable functions that have the 
  same localization coefficient $\alpha_{2,4}(u,\Omega)$.
\end{remark}

\section{A new measure of localization}\label{sec:new}
In the following, $\Omega\subseteq\R^d$ is a domain with a positive Lebesgue measure $|\Omega|$. We further denote by
$\lambda$ the Lebesgue measure, that is, $d\lambda = dx$. Recall that a multiplicative measure $\tau_u$ on $\Omega$ with the associated nonnegative-valued and non-trivial density $u$ is defined by
$d\tau_u =u(x)dx$.

\subsection{Measuring localization using the Wasserstein-2 distance}
Let $\mu$ and $\nu$ be two  Borel probability measures on $\Omega$, and write $W_2(\mu, \nu)$ for the quadratic Wasserstein distance 
between $\mu$ and $\nu$  defined  by \cite{villani2009optimal},
\begin{eqnarray} \label{Transport}
    W_2(\mu,  \nu) =  \left(\inf_{\gamma} \int_{\Omega\times \Omega} |x-y|^2 d\gamma(x,y) \right)^{1/2},
\end{eqnarray}
where the infimum is taken over all Borel probability measures $\gamma$  on $\Omega\times\Omega$ satisfying $\gamma(A\times \Omega) = \mu(A)$ and $\gamma(\Omega\times A) = \nu(A)$ for all Borel subsets $A$ of $\Omega$.

\begin{definition} Let $u$ be the density associated with a Borel probability measure $\mu_u$ on $\Omega$. We then define the localization coefficient $\beta(u,\Omega)$ of $u$ as
\[
\beta(u,\Omega)=W_2(\mu_u,|\Omega|^{-1}\lambda),
\]
where $\lambda$ is the Lebesgue measure having as its density the characteristic function of $\Omega$.
\end{definition}
We thus propose to define and quantify the localization of a probability density $u$ on $\Omega$ in terms of the Wasserstein-2 distance between its associated measure $\mu_u$ and the measure $|\Omega|^{-1}\lambda$ associated with the least localized probability density on $\Omega$, namely with the constant-valued function $|\Omega|^{-1}$. Clearly, we expect more localized $u$ to yield larger values of $\beta(u,\Omega)$. Figure~\ref{fig:beta_good} includes the example from Figure~\ref{fig:alpha_fail} and shows how $\beta(u,\Omega)$ indicates the localization of the family~\eqref{eqn:uad}--\eqref{eqn:j2} of functions, as well as of the family
\begin{equation}\label{eqn:usigma}
u_{\sigma}(x)=\exp(-(x-1/2)^2/\sigma^2)-\exp(-1/4\sigma^2),\quad x\in(0,1).
\end{equation}
Specifically, \(\beta(u,\Omega)\) does distinguish differing degrees of localization among functions sharing the same \(L^p\) norms for \(p \in [1,\infty]\), though it is clearly influenced by a boundary proximity effect. We address the latter in Section~\ref{sec:be}.
\begin{figure}
    \centering
\includegraphics[width=0.49\linewidth]{Images/1D_uds_2.pdf}\includegraphics[width=0.49\linewidth]{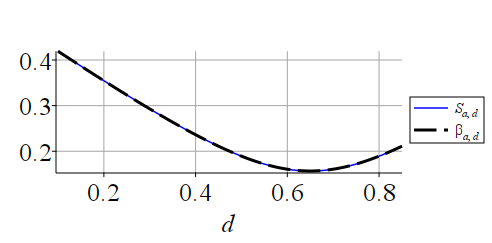}\\\includegraphics[width=0.49\linewidth]{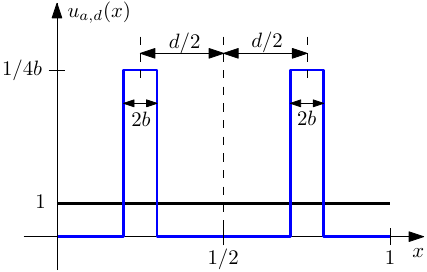}\includegraphics[width=0.49\linewidth]{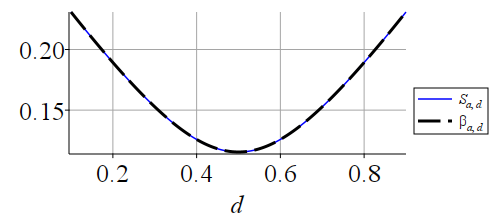}\\\includegraphics[width=0.49\linewidth]{Images/1D_uds_2.pdf}\includegraphics[width=0.49\linewidth]{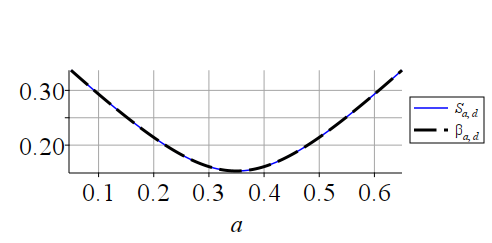}\\\includegraphics[width=0.49\linewidth]{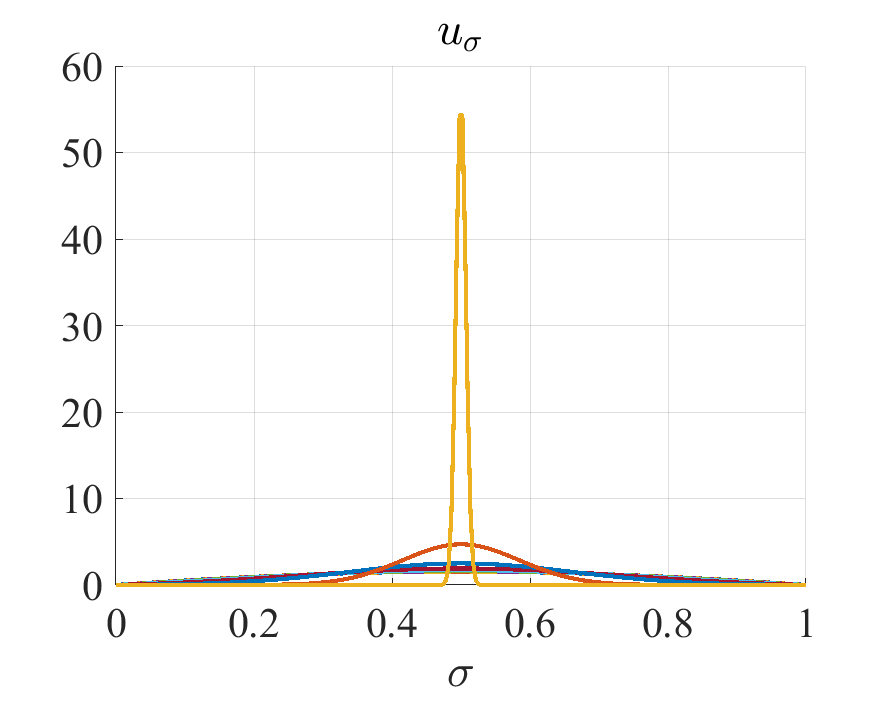}\includegraphics[width=0.49\linewidth]{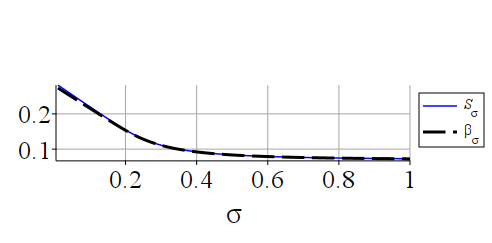}\caption{$\beta_{a,d}=\beta(u_{a,d},(0,1))=W_2(\mu_{u_{a,d}},\lambda)$ vs. $S_{a,d}=\|\mu_{u_{a,d}}-\lambda\|_{\Hdot^{-1}(\lambda)}$ for $u_{a,d}$ from~\eqref{eqn:uad}--\eqref{eqn:j2}, and $\beta_{\sigma}=\beta(u_{\sigma},(0,1))=W_2(\mu_{u_{\sigma}},\lambda)$ vs. $S_{\sigma}=\|\mu_{u_{\sigma}}-\lambda\|_{\Hdot^{-1}(\lambda)}$ for $u_{\sigma}$ from~\eqref{eqn:usigma}: first as function of $d\in[2b,1-a-b]$, with $a=0.1$ and $b=0.05$; next as function of $d\in[b,1/2-b]$, for symmetric densities with $b=0.05$; next as function of $a\in[b,1-d-b]$, when translating a fixed density profile with $b=0.05$ and $d=0.3$; and finally as function of $\sigma\in[0,1]$.}
    \label{fig:beta_good}
\end{figure}

To motivate our new localization concept we shall show how the Wasserstein-2 distance is related to the optimal transport between measures. Let $T: \Omega \to \Omega$ be a Borel measurable mapping. The push-forward of a probability measure $\mu$  by the mapping $T$ is the probability measure, denoted by $T_\#\mu$, and defined by $T_\#\mu(B) = \mu(T^{-1}(B)) $ for all Borel subsets $B$ of $\Omega$. $T_\#$ is also called a transport map from $\mu$ to $\nu$. We recall the result on the existence and characterization of a minimizer of \eqref{Transport} under additional assumptions on 
$\mu$ due to Brenier \cite{brenier1991polar}. If $\mu$ and $\nu$ have finite moments of order $2$, and $\mu$
is absolutely continuous with respect to  the Lebesgue measure ($d\mu(x) = \rho(x) dx$ with $\rho $ a non-negative Lebesgue integrable function), then
\begin{eqnarray} \label{Transport2}
    W_2(\mu,  \nu) =  \left(\inf_{T_\#\mu =\nu} \int_{\Omega} |x-T(x)|^2 \mu(x) dx \right)^{1/2},
\end{eqnarray}
where the infimum is calculated over all measurable Borel mappings $T: \Omega \to \Omega$ that push forward $\mu$ onto $\nu$. Notice that 
 since  $T_\#\mu =\nu$ the joint measure $(\textrm{Id},T)_{\#} \mu$ on $\Omega\times \Omega$ lies in the minimization set 
for \eqref{Transport}, and hence  $W_2(\mu,  \nu)$ is by construction lower than the infimum of \eqref{Transport2}.  The reverse inequality
is  more difficult to show and is  known to not hold in a more general setting \cite{villani2009optimal, santambrogio2015optimal, figalli2021invitation}.\\

\subsection{Localization in terms of a Sobolev norm}

Using the Wasserstein-2 distance to measure localization is intuitive and well-justified, but the actual computation of the localization coefficient \( \beta(u,\Omega) \) then requires solving a linear programming problem. For a density \( u \) sampled at \( n \) grid points in a bounded domain \( \Omega\subset\R^d \), the computational complexity is \( O(n^3 + n^2d) \), accounting for the calculation of pairwise distances between grid points. Typically, \( n \sim N^d \) for some constant \( N\gg1 \), which implies that computing \( W_2(\mu_u, \nu_{\text{deloc}}) \) from Definition~\ref{def:1} typically has a complexity of \( O(N^{3d} + N^{2d}d) \). As the dimension \(d\) grows, this scaling becomes impractical and can introduce a significant computational bottleneck—particularly in iterative procedures aiming to optimize a function’s localization by adjusting model parameters and repeatedly evaluating \(\beta(u,\Omega)\). To address this issue, we can leverage a linearization of the Wasserstein-2 distance $W_2(\mu_u,|\Omega|^{-1}\lambda)$ in terms of the weighted homogeneous Sobolev norm $\|\mu_u-|\Omega|^{-1}\lambda\|_{\Hdot^{-1}(|\Omega|^{-1}\lambda)}$, defined for any signed measure $\tau_t\in \dot{H}^{-1}(|\Omega|^{-1}\lambda)$ with zero-mean density $t(x)$ as
\[
\|\tau_t\|_{\Hdot^{-1}(|\Omega|^{-1}\lambda)}=\sup\left\{\left|\int_{\Omega}\varphi(x)t(x)dx\right|,\,\varphi\in\dot{H}^1(|\Omega|^{-1}\lambda),\,\|\nabla\varphi\|_{L^2(|\Omega|^{-1}\lambda)}\le 1\right\}.
\]
Specifically, equipping the space $\dot{H}^1(|\Omega|^{-1}\lambda)$ of zero-mean $H^1(\Omega)$ functions with the inner product $(\varphi,\psi)=(\nabla\varphi,\nabla\psi)_{L^2(|\Omega|^{-1}\lambda)}$ and with the induced norm
\[
\|\varphi\|^2_{\Hdot^1(|\Omega|^{-1}\lambda)}=\int_{\Omega}|\nabla\varphi(x)|^2|\Omega|^{-1}dx,
\]
as well as integrating by parts, we see from Riesz's representation theorem that there is a unique $w\in \dot{H}^1(|\Omega|^{-1}\lambda)$ satisfying
\[
\int_{\Omega}\varphi(x)t(x)dx=\langle\varphi,\tau_t\rangle=\int_{\Omega}\nabla w(x).\nabla\varphi(x)|\Omega|^{-1}dx,\quad\varphi\in \dot{H}^1(|\Omega|^{-1}\lambda),
\]
which in turn is the weak formulation of the boundary value problem
\begin{equation}\label{eqn:PoiNeu}
\left\{\begin{array}{rcl}-\Delta w&=&|\Omega|t,\quad\text{in}\,\,\Omega,\\\partial_nw&=&0\quad\text{at}\,\,\partial\Omega.\end{array}\right.
\end{equation}
Since the compatibility condition
\[
\int_{\Omega}t(x)dx=0,
\]
is satisfied by assumption, we have a unique weak solution of~\eqref{eqn:PoiNeu} in $H^1(\Omega)$. Thus, the unique representative of $\tau_t$ in $\dot{H}^1(|\Omega|^{-1}\lambda)$ satisfies~\eqref{eqn:PoiNeu} in the weak sense. It is furthermore readily shown that $\|\tau_t\|_{\Hdot^{-1}(|\Omega|^{-1}\lambda)}=|\Omega|^{-1/2}\|\nabla w\|_{L^2(\Omega)}$: indeed, we have
\begin{align*}
\|\tau_t\|_{\Hdot^{-1}(|\Omega|^{-1}\lambda)}&\le\sup\left\{\int_{\Omega}|\nabla\varphi||\nabla w||\Omega|^{-1}dx,\,\, \varphi\in\dot{H}^1(|\Omega|^{-1}\lambda),\,\|\nabla\varphi\|_{L^2(|\Omega|^{-1}\lambda)}\le 1\right\}\\&\le|\Omega|^{-1/2}\|\nabla w\|_{L^2(\Omega)},
\end{align*}
while
\[
\left|\left\langle w|\Omega|^{1/2}\|\nabla w\|_{L^2(\Omega)}^{-1},\tau_t\right\rangle\right|=|\Omega|^{-1/2}\|\nabla w\|_{L^2(\Omega)}.
\]
We can in fact connect the Wasserstein-2 distance and the Sobolev $\Hdot^{-1}$ metric. Starting with dimension one, given an interval $\Omega=(x_0,x_1)$, we first transform $L^1(\Omega)$-normalized densities $v$ to $L^1(0,1)$-normalized densities $u$ via $u(x)=|\Omega|v(|\Omega|x+x_0)$, $x\in(0,1)$. Next, let $f(x)=x$, $x\in(0,1)$, be the cumulative density function of the constant-valued density $1$, with the inverse function $f^{-1}(y)=y$, $y\in(0,1)$. We have that, explicitly~\cite{ramdas2017wasserstein},
\begin{equation}\label{eqn:W2_1D}
W_2(\mu_u,\lambda)^2=\int_0^1(U^{-1}(y)-f^{-1}(y))^2dy,
\end{equation}
where $U$ is the cumulative distribution function (CDF) of $u$,
\[
U(x)=\int_0^xu(t)dt,\quad x\in(0,1),
\]
and $U^{-1}$ is the corresponding quantile function. Since $U$ may generally be constant in some regions, and thus non-injective, we need to define $U^{-1}$ as a pseudo-inverse for use in~\eqref{eqn:W2_1D},
\begin{equation*}
    U^{-1}(r) \;=\; \inf\left\{t: \;\; U(t)>r \right\}.
\end{equation*}
We now establish the equality of the Wasserstein-2 distance and the Sobolev $\Hdot^{-1}$ metric in dimension one:
\begin{lemma}\label{lem:W2_vs_H-1} For any probability density $u$ on $(0,1)$ we have
\[
\|\mu_u-\lambda\|_{\Hdot^{-1}(\lambda)}=W_2(\mu_u,\lambda).
\]
\end{lemma}
\begin{proof}
The system
\begin{equation*}\left\{
    \begin{array}{rcl}-g''(x)&=&u(x)-1,\quad x\in(0,1),\\
    g'(0)&=&0,\\
    g'(1)&=&0,\\
    \int_0^1g(x)dx&=&0,\end{array}\right.
\end{equation*}
has the unique solution
\begin{align*}
g(x)=-\int_{0}^x\int_{0}^t(u(\tau)-1)d\tau dt+x\int_0^1(u(t)-1)dt+C,\quad x\in(0,1),
\end{align*}
with some constant $C$, and thus
\begin{align*}
-g'(x)&=\int_0^x(u(t)-1)dt-\int_0^1(u(t)-1)dt\\&=U(x)-f(x),\quad x\in(0,1),
\end{align*}
leading to 
\[
\|\mu_u-\lambda\|_{\Hdot^{-1}(\lambda)}^2=\int_0^1g'(x)^2dx=\int_0^1(U(x)-x)^2dx.
\]
To relate this to the Wasserstein-2 distance between $\mu_u$ and $\lambda$, we first assume that $u\in C([0,1])$. Then $\Sigma_u = u^{-1}((0, +\infty)) \cap (0,1)$ is a non-empty open set, and we claim that $U: \Sigma_u\rightarrow U(\Sigma_u)$ is bijective.
To prove that $U|_{\Sigma_u}$ is injective, assume that $U(x_1) = U(x_2)$ for some $x_1, x_2 \in \Sigma_u$ satisfying $x_1 < x_2$. Since $U$ is non-decreasing we have $U(x) = U(x_1)$ for all $x\in [x_1, x_2]$. Hence $u(x) = 0 $ on $(x_1, x_2)\cap \Sigma_u \not= \emptyset,$ which is in contradiction with the definition of  $\Sigma_u$. Now write $\widetilde U^{-1}$ for the inverse function of $U|_{\Sigma_u}$. We shall now show that $\widetilde U^{-1}=U^{-1}$ on $U(\Sigma_u)$.  Let $r= U(x_r)$ with $x_r = \widetilde U^{-1}(r) \in \Sigma_u$. Since $\Sigma_u$ is an open set there exists $\varepsilon>0$ small enough such that 
$(x_r-\varepsilon, x_r+\varepsilon) \subset \Sigma_u$. Taking into account that $U$ is increasing  on  $(x_r-\varepsilon, x_r+\varepsilon)$, we have $U^{-1}(r) = x_r= \widetilde U^{-1}(r)$. Consequently $U^{-1}(U(x)) = x$ on $\Sigma_u$. With this in mind, we can perform the change of variables $y=U(x)$ and write
\begin{eqnarray*}
    W_2(\mu_u,\lambda)^2&=&\int_0^1(U^{-1}(y)-y)^2dy=\int_0^1\left(U^{-1}(U(x))-U(x)\right)^2u(x)dx\\&=&\int_{\Sigma_u}\left(U^{-1}(U(x))-U(x)\right)^2u(x)dx = \int_{\Sigma_u}(x-U(x))^2u(x)dx \\&=& \int_0^1(U(x)-x)^2(u(x)-1)dx+ \int_0^1(U(x)-x)^2dx,
\end{eqnarray*}
where finally, via integration by parts,
\[
\int_0^1(U(x)-x)^2\left(u(x)-1\right)dx=0,
\]
that is,
\[
W_2(\mu_u,\lambda)^2 = \|\mu_u-\lambda\|_{\Hdot^{-1}(\lambda)}^2.
\]
For a general probability density $u$, the result follows from the denseness of $C([0,1])$ in $L^1([0,1])$.
\end{proof}

\begin{remark}
   Actually one can check that  we have  $\overline{U(\Sigma_u)} = [0,1]$.  Assume that the latter is false, and let $y\in [0,1] \setminus \overline{U(\Sigma_u)}$.  
We first consider the case $y>0$, and  set $x^* = \inf\{x\in [0,1]; \;\; U(x)=y \}$. Since $U$ is continuous we have $U(x^*)=y$. Recall that $U$ is non-decreasing, and so $\{x\in [0,1]; \;\; U(x)=y \}$ is non-empty. Using the property of the infimum and the fact that $U(x) = \int_{\Sigma_u\cap (0,x)} u(t) dt$, one can easily show that there exists $\varepsilon>0$ small enough such that $(x^*-\varepsilon, x^*) \subset \Sigma_u$. Thus the sequence $y_j = U(x^*-\frac{1}{j}) \in U(\Sigma_u)$ converges to $y$, which is in contradiction with the
first assumption. If $y=0$, the same result can be derived  by taking $x^* = \sup\{x\in [0,1]; \;\; U(x)=y \}$ and proving that $(x^*, x^*-\varepsilon) \subset \Sigma_u$
for some small $\varepsilon>0$.
\end{remark}
Figure~\ref{fig:beta_good} is a numerical illustration of the result of Lemma~\ref{lem:W2_vs_H-1} for particular families of probability densities.

In dimension $d>1$ with the density $u$ bounded, we readily specialize Peyre~\cite[theorems 1 and 5]{Peyre-2018} to estimates relating $W_2(\mu_u,|\Omega|^{-1}\lambda)$ and $\|\mu_u-|\Omega|^{-1}\lambda\|_{\Hdot^{-1}(|\Omega|^{-1}\lambda)}$. Indeed, $\mu_u\le\|u\|_{L^{\infty}(\Omega)}\lambda$, so we can follow the proof of~\cite[Theorem 5]{Peyre-2018} with the additional observation that $\mu_t\le(\|u\|_{\infty}|\Omega|)^{1-t}|\Omega|^{-1}\lambda$, where $\mu_t$, $t\in[0,1]$, is the displacement interpolation between $\mu_u$ and $|\Omega|^{-1}\lambda$ satisfying $\mu_0=\mu_u$ and $\mu_1=|\Omega|^{-1}\lambda$, to conclude that
\begin{equation}\label{eqn:W2lower}
\frac{\ln(\|u\|_{\infty}|\Omega|)}{2(\|u\|_{\infty}^{1/2}|\Omega|^{1/2}-1)}\|\mu_u-|\Omega|^{-1}\lambda\|_{\Hdot^{-1}(|\Omega|^{-1}\lambda)}\le W_2(\mu_u,|\Omega|^{-1}\lambda);
\end{equation}
note that we here use the fact that $\|\cdot\|_{\Hdot^{-1}(\lambda)}=|\Omega|^{-1/2}\|\cdot\|_{\Hdot^{-1}(|\Omega|^{-1}\lambda)}$.
Finally~\cite[Theorem 1]{Peyre-2018} gives the upper bound
\begin{equation}\label{eqn:W2upper}
W_2(\mu_u,|\Omega|^{-1}\lambda)\le2\|\mu_u-|\Omega|^{-1}\lambda\|_{\Hdot^{-1}(|\Omega|^{-1}\lambda)}.
\end{equation}
We expect that the estimates~\eqref{eqn:W2lower}--\eqref{eqn:W2upper} can be improved in the present special case where the measures are defined in terms of probability densities. Specifically, since the $\Hdot^{-1}$-norm is related to the regularity of the underlying densities 
and the Wasserstein-2 distance measures the optimal transportation cost between the densities, we expect the optimal bounds to depend on the  geometry of the domain $\Omega$.

\subsection{Compensating for the boundary effect}\label{sec:be}
If the domain \(\Omega\) has a boundary then the intuitive concept of localization for a density \(u\) defined on \(\Omega\) may not align directly with the numerical value \(\beta(u,\Omega)\). The first three cases illustrated in Figure~\ref{fig:beta_good} exemplify this discrepancy: the value of \(\beta\) increases as portions of the mass of \(u\) approach the boundary of the interval \((0,1)\), even though intuitively one might perceive \(u\) as maintaining or even losing localization in such situations. Since it is indeed relatively expensive to redistribute a mass concentrated near the boundary back to the uniform mass distribution $|\Omega|^{-1}$, this estimation of localization does make sense; however, if the examples of Figure~\ref{fig:beta_good} come from a numerical truncation of an unbounded domain, then clearly the artificially introduced boundary distorts our notion of localization of $u$. One way to alleviate this issue is to allow the flow of mass across $\partial\Omega$, for example by periodizing the mass transport problem. This, however, is challenging for general geometries $\Omega$ in $\R^d$ with $d\ge2$, and it does not always work satisfactorily, as illustrated in Figure~\ref{fig:periodic_bad}. There, the cost function
\[
c(x,y)=(x-y)^2,\quad x,y\in(0,1),
\]
associated with the Wasserstein-2 distance, is replaced with the periodized cost function
\[
c(x,y)=\min\{(x-y)^2,(1-|x-y|)^2\},\quad x,y\in(0,1).
\]
\begin{figure}
    \centering
    \includegraphics[width=0.48\linewidth]{Images/1D_uds_2.pdf}
    \includegraphics[width=0.51\linewidth]{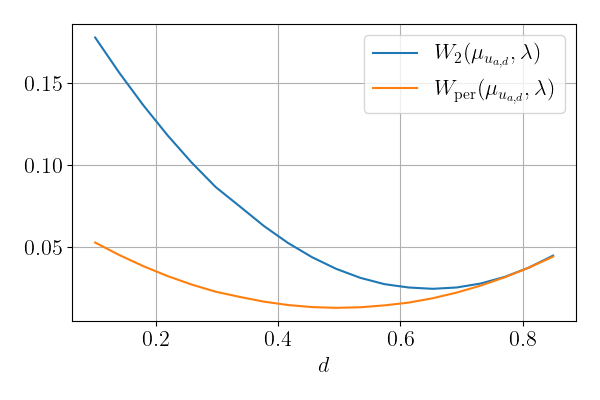}
    \caption{Wasserstein-2 optimal transport cost vs. periodized optimal transport cost between $u_{a,d}$ of~\eqref{eqn:uad}--\eqref{eqn:j2} and the constant-valued function $1$, as function of $d\in[2b, 1 - a - b]$, with $a = 0.1$ and $b = 0.05$.}
    \label{fig:periodic_bad}
\end{figure}
A simpler approach is to replace the domain $\Omega$ with a larger domain $\Omega'$, assume the density $u$ to vanish in $\Omega'\setminus\Omega$, and compute the optimal transport cost from $\mu_u$ to the constant-valued measure $|\Omega'|^{-1}\lambda$. This generally works well, as illustrated in Figure~\ref{fig:extension_good}, but can be challenging especially in high dimension due to a potentially large number of grid points needed to describe the larger domain $\Omega'$.
\begin{figure}
    \centering
    \includegraphics[width=0.5\linewidth]{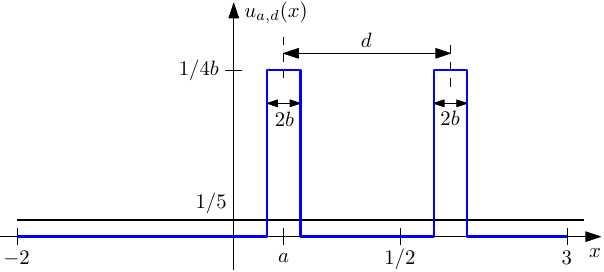}
    \includegraphics[width=0.51\linewidth]{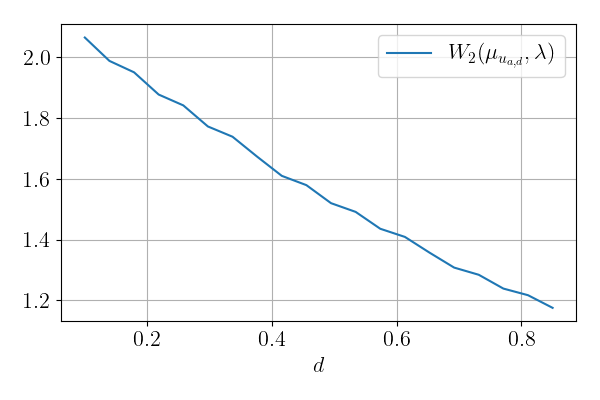}
    \caption{Extended-domain Wasserstein-2 optimal transport cost between $u_{a,d}$ of~\eqref{eqn:uad}--\eqref{eqn:j2} and the constant-valued function $1/5$, as function of $d\in[2b, 1 - a - b]$, with $a = 0.1$ and $b = 0.05$.}
    \label{fig:extension_good}
\end{figure}
Finally, the cost function can be modified to compensate for the presence of the boundary by attaining lower values for points close to $\partial\Omega$. We found this approach to work poorly in dimension one, ostensibly due to the lack of real boundary geometry there, but it seems to perform well in dimension two, as illustrated in Section~\ref{sec:dtb2D}.

\section{Further numerical examples of $\beta(u,\Omega)$}\label{sec:numerical}
\subsection{Localization of eigenfunctions of a Sturm-Liouville operator}\label{sec:leslo}
We now revisit the eigenvalue problem from~\cite{2024-localization_1D},
\begin{equation}\label{eqn:evp}
\left\{\begin{array}{rcl}
-(p(x)\phi'_{\lambda}(x))'&=&\lambda\phi_{\lambda}(x),\quad x\in(0,1),\\
\phi_{\lambda}(0)&=&\phi_{\lambda}(1)=0,
\end{array}\right.
\end{equation}
for the Dirichlet Laplacian on $(0,1)$ with the metric
\[
p(x)=\tanh(40x-10)+1.1,\quad x\in(0,1),
\]
as illustrated in Figure~\ref{fig:p}.
\begin{figure}
    \centering
    \includegraphics[width=0.5\linewidth]{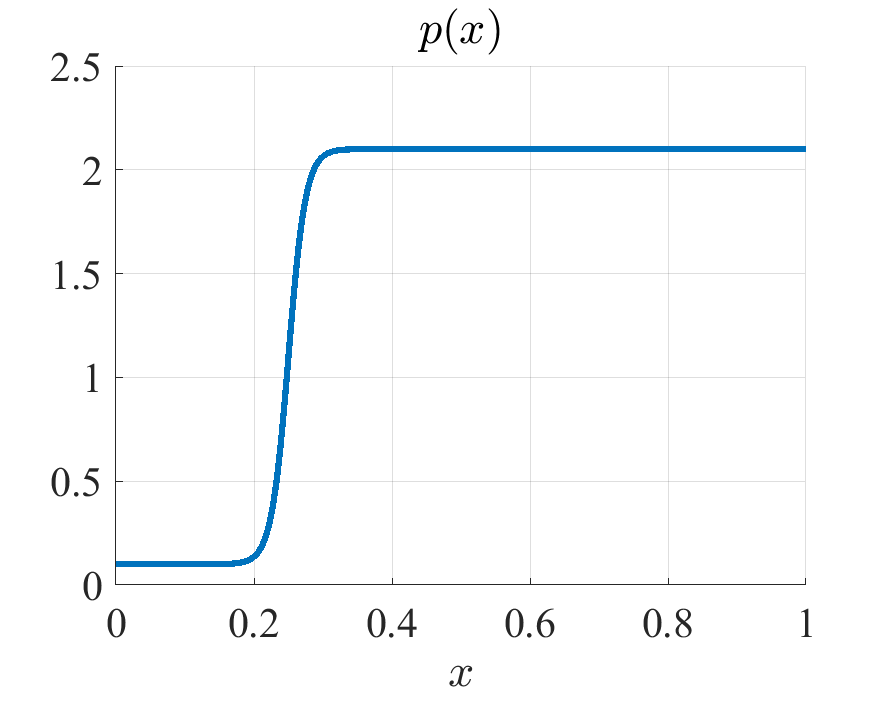}
    \caption{The metric for the Laplacian of the numerical example in Section~\ref{sec:leslo}.}
    \label{fig:p}
\end{figure}
Figure~\ref{fig:alpha_bar_beta_bar} shows a comparison between the localization coefficients $\alpha^{-1}_{2,4}(|\phi_{\lambda_i}|/\|\phi_{\lambda_i}\|_{L^1(0,1)},(0,1))$ and $\beta(|\phi_{\lambda_i}|/\|\phi_{\lambda_i}\|_{L^1(0,1)},(0,1))$, $i=1,\dots,40$, where a normalization is performed for easier reference.
\begin{figure}
    \centering
    \includegraphics[width=0.45\linewidth]{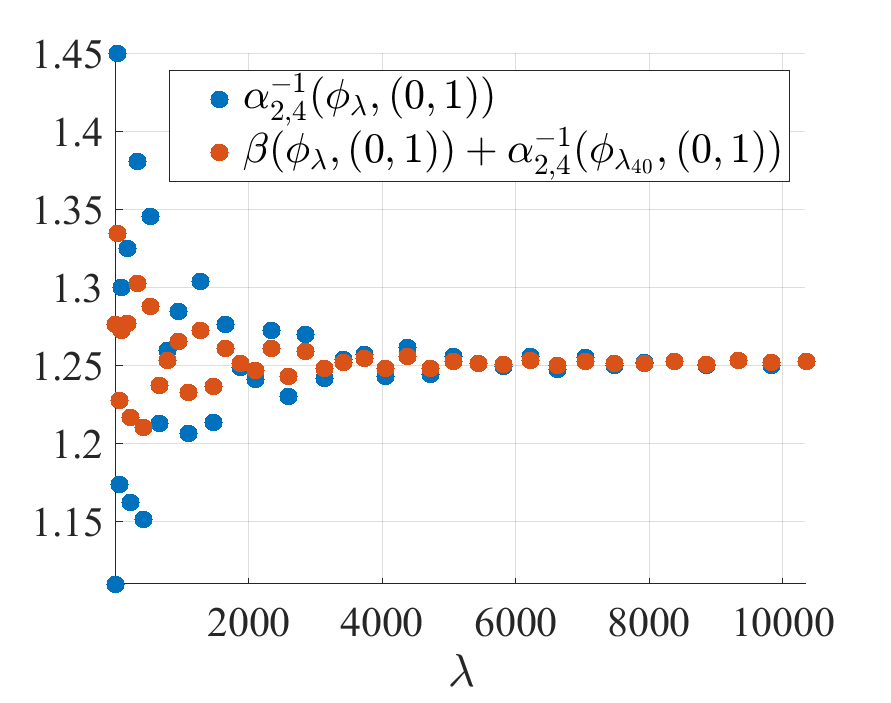}
    \includegraphics[width=0.45\linewidth]{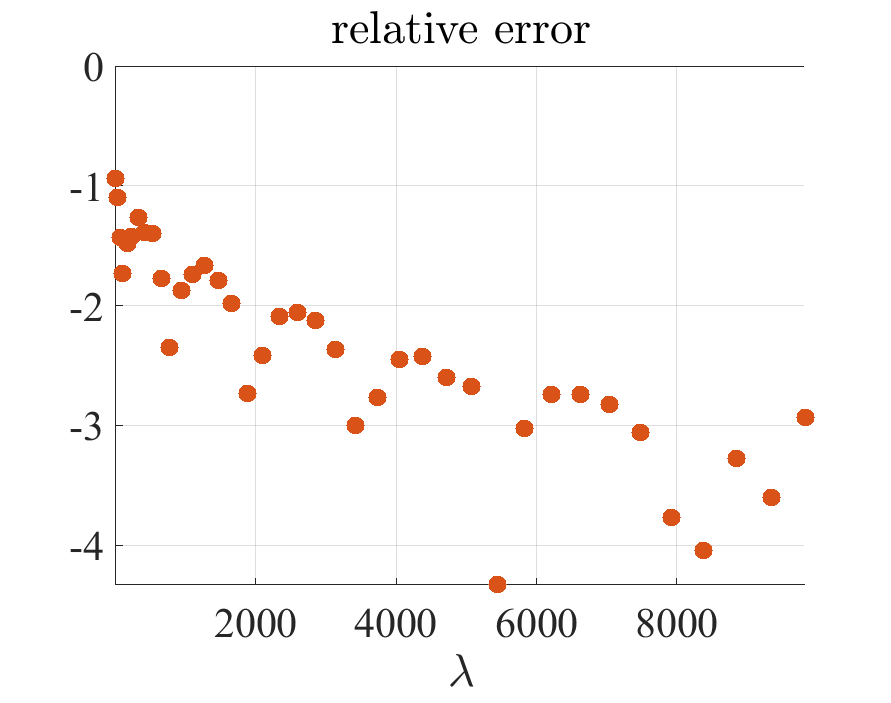}\\
    \includegraphics[width=0.45\linewidth]{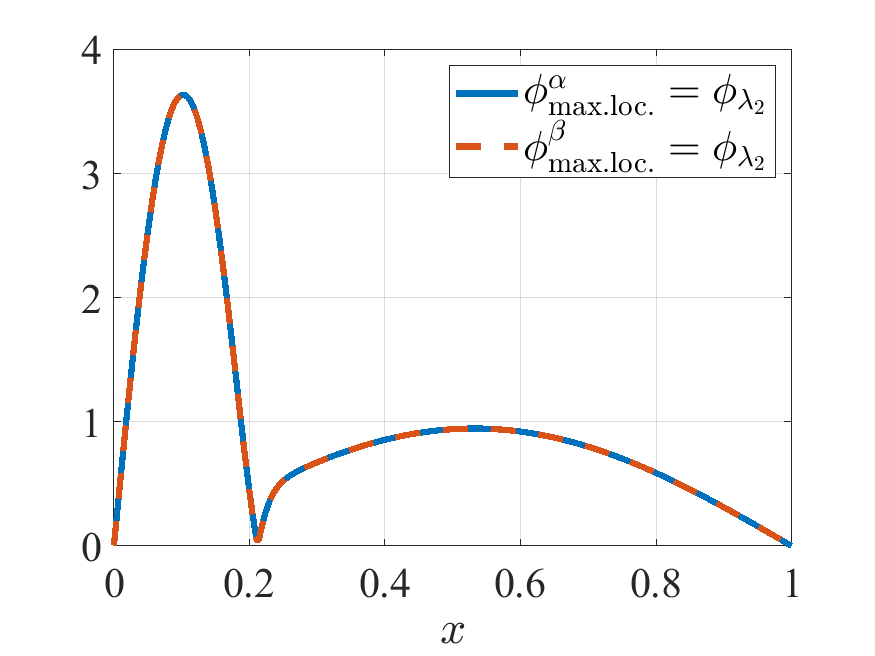}
    \includegraphics[width=0.45\linewidth]{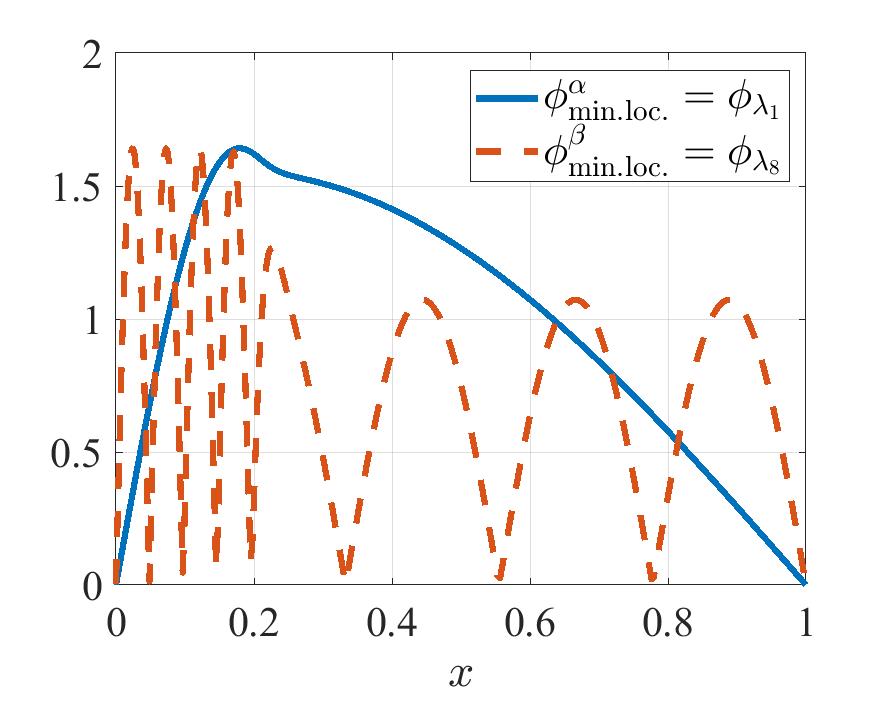}
    \caption{Top: The localization coefficients $\alpha^{-1}_{2,4}(\phi_{\lambda},(0,1))$ and $\beta(\phi_{\lambda},(0,1))+\alpha_{2,4}^{-1}(\phi_{\lambda_{40}},(0,1))$, here normalized for the purposes of comparison, coincide well overall for the eigenfunctions of the Sturm-Liouville operator considered in Section~\ref{sec:leslo}. Bottom: However, $\beta(\phi_{\lambda},(0,1))$ detects the extrema of localization more robustly than $\alpha_{2,4}(\phi_{\lambda},(0,1))$.}
    \label{fig:alpha_bar_beta_bar}
\end{figure}
There is evidently a good correspondence in the qualitative behavior of $\alpha_{2,4}$ and $\beta$ as functions of the eigenvalue $\lambda$. Specifically, both localization coefficients identify the $L^1$-normalized modulus of the second eigenfunction $\phi_{\lambda_2}$ as the most localized out of the first 40 eigenfunctions considered. Figure~\ref{fig:exten} shows the values $\beta(|\phi_{\lambda_i}|/\|\phi_{\lambda_i}\|_{L^1(-1,2)},(-1,2))$, that is, it shows the effect of extending the domain $\Omega$. While there are still significant local oscillations, the delocalization as the index $i$ grows is now better pronounced.
\begin{figure}
    \centering
    \includegraphics[width=0.8\linewidth]{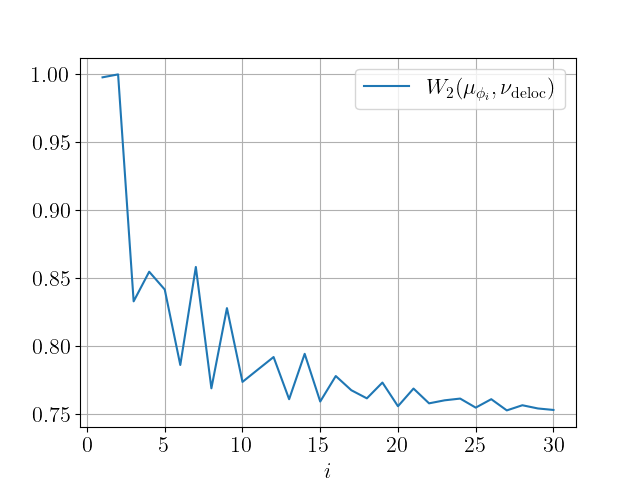}
    \caption{The $W_2$ measure of localization for the eigenvectors $\phi_i$ over the extended domain $[-1,2]$.}
    \label{fig:exten}
\end{figure}

\subsection{Localization of Neumann-Poincaré eigenfunctions on curves}\label{sec:efs_curve}
As another numerical example, we consider the localization of eigenfunctions of the Neumann-Poincaré operator associated with transmission problems of the form
\begin{equation}\label{eqn:MAIN}
\left\{\begin{array}{rcl}\Delta u&=&0\quad{\rm in}\,\,\R^2\setminus\partial\Omega,\\u|_{\partial\Omega+}&=&u|_{\partial\Omega-},\\\varepsilon\partial_{\nu}u|_{\partial\Omega+}&=&-\partial_{\nu}u|_{\partial\Omega-},\end{array}\right.
\end{equation}
where $\Omega\subset\R^2$ is an open bounded set, star-like with respect to the origin in $\R^2$, and where the boundary $\partial\Omega$ of $\Omega$ is a smooth closed curve. Moreover, $\varepsilon\in\R\setminus\{-1\}$ is a constant, $(\cdot)|_{\partial\Omega+}$ and $(\cdot)|_{\partial\Omega-}$ are limits at $\partial\Omega$ taken from $\Omega$ and from $\R^2\setminus\overline\Omega$, respectively, and $\nu$ is the unit outward normal to $\partial\Omega$. It is well-known~\cite[Proposition 11.3 on p. 40]{TaylorII} that setting $u=\mathcal{S}\phi$ in~\eqref{eqn:MAIN}, where $\mathcal S$ is the single layer potential at $\partial\Omega$,
\[
\mathcal S\phi(x)=\int_{\partial\Omega}\Phi(x-y)\phi(y)d\ell(y),\quad x\in\R^2\setminus\partial\Omega,
\]
and $\Phi$ is the outgoing fundamental solution of the Laplacian in $\R^2$,
\[
\Phi(x)=-\frac{1}{2\pi}\ln|x|,\quad x\in\R^2\setminus\{0\},
\]
leads to the eigenvalue problem $N^{\#}\phi=(\varepsilon-1)(\varepsilon+1)^{-1}\phi$ at $\partial\Omega$, with the Neumann-Poincaré operator $N^{\#}$ given by
\begin{equation}\label{eqn:Nhash}
N^{\#}\phi(x)=2\,{\rm p.v.}\int_{\partial\Omega}\partial_{\nu_x}\Phi(x-y)\phi(y)d\ell(y)=-\frac{1}{\pi}\,{\rm p.v.}\int_{\partial\Omega}\frac{(x-y).\nu_x}{|x-y|^2}\phi(y)d\ell(y),\quad x\in\partial\Omega.
\end{equation}
Thus, fixing the boundary $\partial\Omega$ and computing the eigenvalues $\lambda$ and eigenfunctions $\phi$ of $N^{\#}$ from~\eqref{eqn:Nhash} corresponds to solving transmission problems~\eqref{eqn:MAIN} with $\varepsilon=(1+\lambda)/(1-\lambda)$. Now, given the parametrization
\[
\partial\Omega=\omega([0,1))=\{(r(\theta)\cos(2\pi\theta),r(\theta)\sin(2\pi\theta)),\,\,\theta\in[0,1)\},
\]
the action of a measure $\mu_v$ with density $v$ on $\partial\Omega$ on any test function $\varphi$ on $\partial\Omega$ is expressible in terms of
\[
\langle\varphi,\mu_v\rangle=\int_{\partial\Omega}\varphi(x)v(x)ds(x)=\int_0^1\omega^{\ast}\varphi(t)\omega^{\ast}v(t)|\omega'(t)|dt,
\]
that is, in terms of the action of the measure with the pulled-back density $|\omega'|\omega^{\ast}v$, defined in $[0,1)$, on the pullback of the test function. We are therefore interested in the localization coefficients $\alpha_{2,4}(\widetilde\phi,(0,1))$ and $\beta(\widetilde\phi,(0,1))=\|\mu_{\widetilde\phi}-\lambda\|_{\Hdot^{-1}(\lambda)}$
of the normalized pullback
\[
\widetilde\phi=|\omega'|\cdot|\omega^{\ast}\phi|/\||\omega'|\omega^{\ast}\phi\|_{L^1(0,1)}
\]
of eigenfunctions $\phi$ of $N^{\#}$. For our numerical examples we consider three curves $\partial\Omega$ produced by the superformula~\cite{Gielis-2003},
\begin{equation}\label{eqn:super}
r(\theta)=\left(|a|^{-n_2}|\cos(\pi m\theta/2)|^{n_2}+|b|^{-n_3}|\sin(\pi m\theta/2)|^{n_3}\right)^{-1/n_1},\quad\theta\in[0,1),
\end{equation}
as shown in Figure~\ref{fig:curves}.
\begin{figure}
    \centering
    \includegraphics[width=0.325\linewidth]{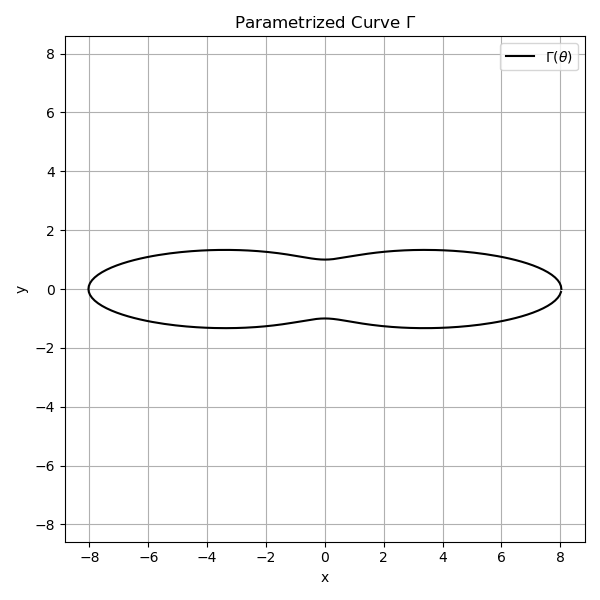}
    \includegraphics[width=0.325\linewidth]{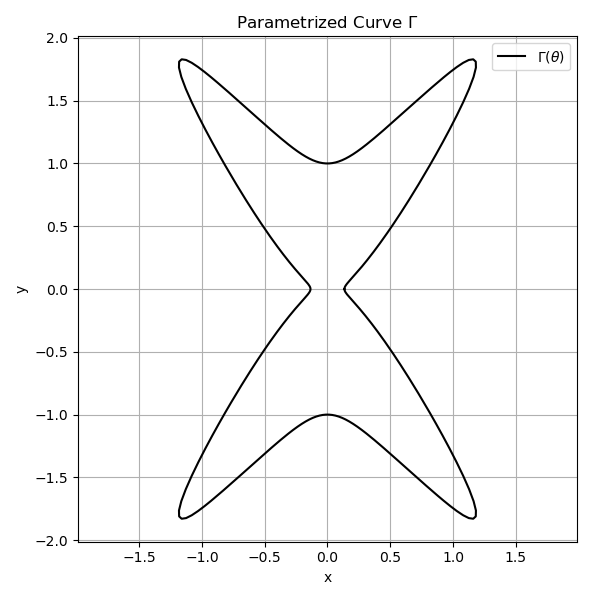}
    \includegraphics[width=0.325\linewidth]{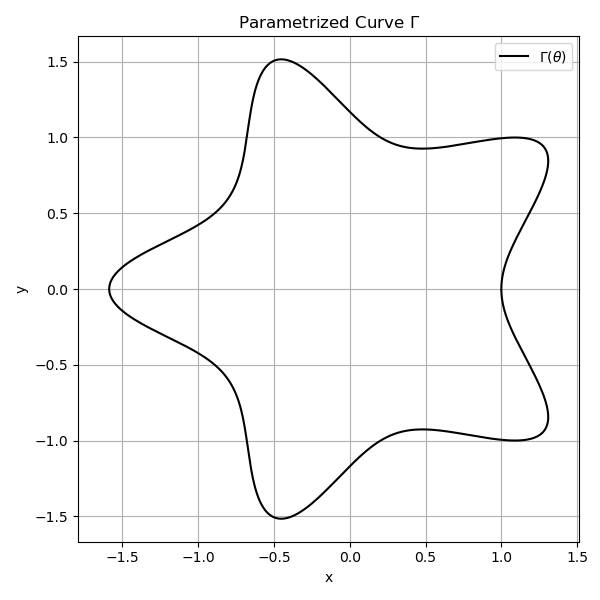}
    \caption{Three different curves $\partial\Omega$ for the numerical examples of Section~\ref{sec:efs_curve}. The curves are parametrized by the superformula~\eqref{eqn:super}. Left: $(m,n_1,n_2,n_3,a,b)=(4, 1.4, 8, 2,1.44,1)$. Center: $(m,n_1,n_2,n_3,a,b)=(4, 4, 20, 20, 0.67, 1)$. Right: $(m,n_1,n_2,n_3,a,b)=(5,3,6,6,1,1)$.}
    \label{fig:curves}
\end{figure}
Figure~\ref{fig:betas_and_one_over_alphas} shows the computed $\beta(\widetilde\phi_{\lambda},(0,1))$ and $\alpha_{2,4}^{-1}(\widetilde\phi_{\lambda},(0,1))$ as function of $\lambda$ for the $L^1$-normalized eigenfunctions pulled back from the three curves.
\begin{figure}
    \centering
    \includegraphics[width=0.75\linewidth]{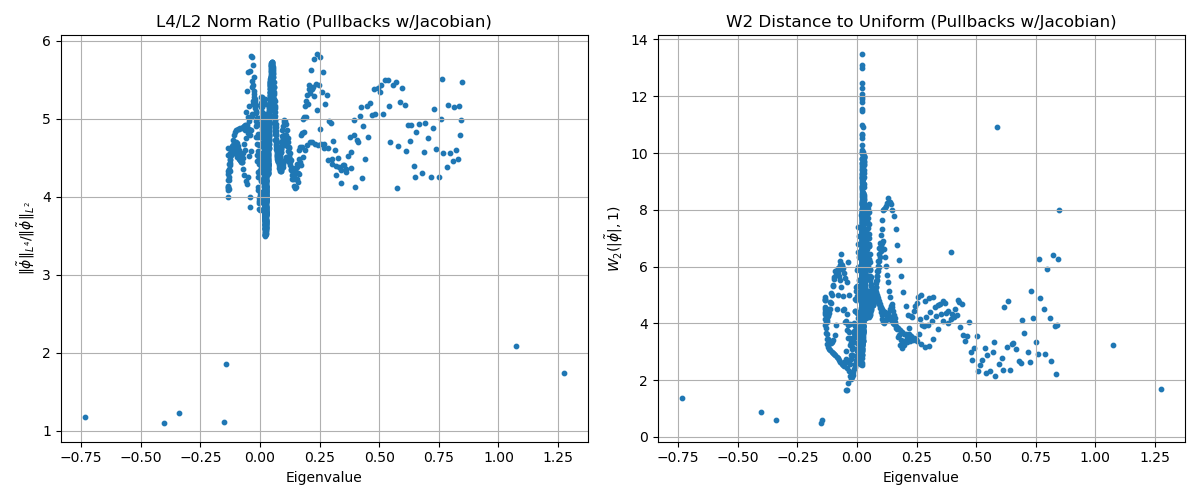}
        \includegraphics[width=0.75\linewidth]{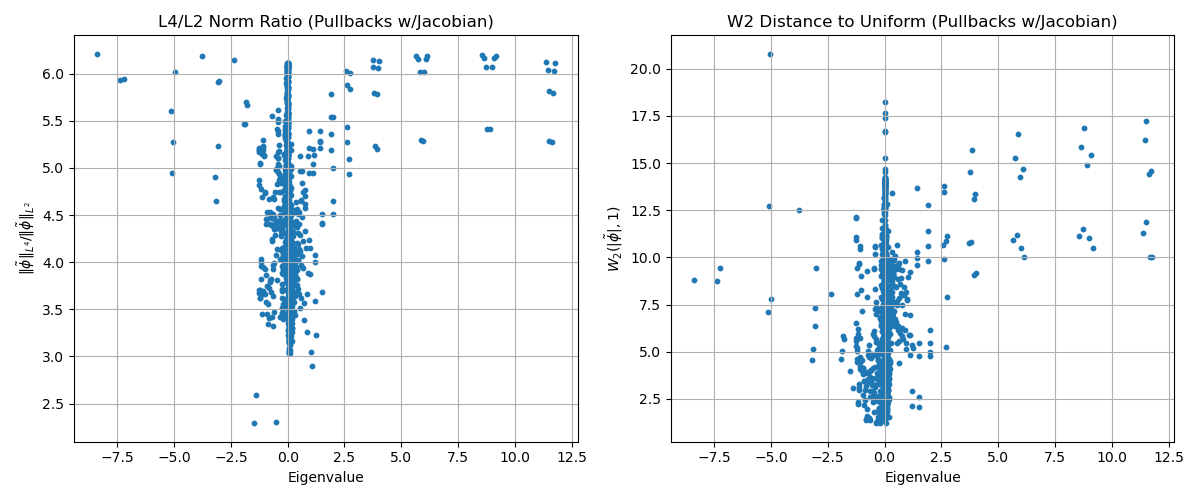}
   \includegraphics[width=0.75\linewidth]{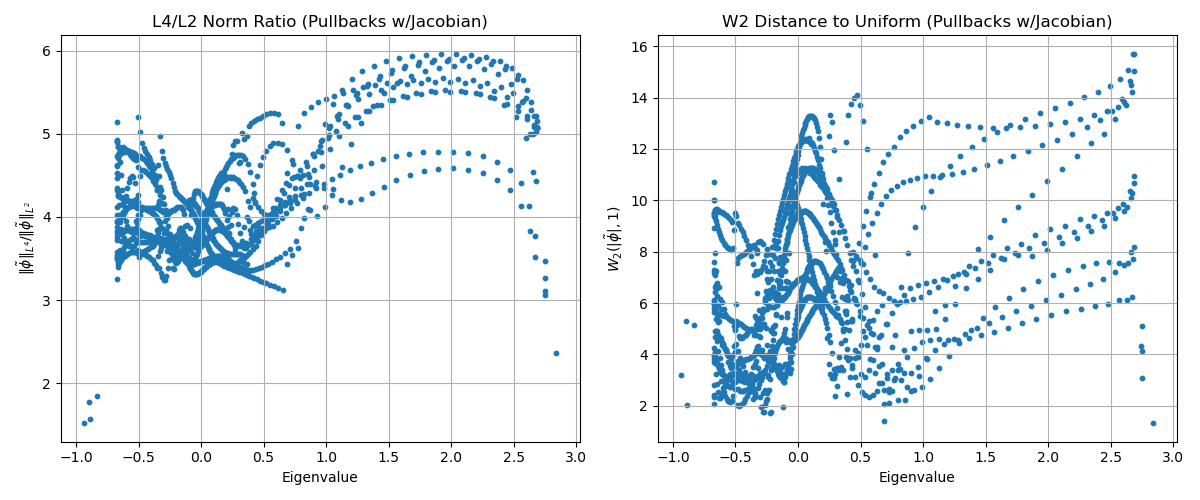}
    \caption{Localization coefficients $\alpha_{2,4}(\widetilde\phi_{\lambda},(0,1))^{-1}$ and $\beta(\widetilde\phi_{\lambda},(0,1))$ of eigenfunctions of the Neumann-Poincaré operator on curves from Figure~\ref{fig:curves}.}
    \label{fig:betas_and_one_over_alphas}
\end{figure}
We see that the behavior of our localization coefficient $\beta$ is quite similar to that of $\alpha_{2,4}^{-1}$. It is well-known that the eigenspace corresponding to the zero eigenvalue contains arbitrarily localized functions, which is consistent with the asymptotic behavior of both $\beta$ and $\alpha_{2,4}^{-1}$ as $\lambda$ approaches zero \cite{ammari2023quantum}. In Figure~\ref{fig:eigvals_and_coeffs}, we see that $\beta$ traverses a larger interval than $\alpha_{2,4}^{-1}$, and is especially better at distinguishing the first few most localized eigenfunctions. Finally, as expected, Figure~\ref{fig:3D_views} indicates that the most localized eigenfunctions localize predominantly near high-curvature points of $\partial\Omega$.
\begin{figure}
    \centering
    \includegraphics[width=0.49\linewidth]{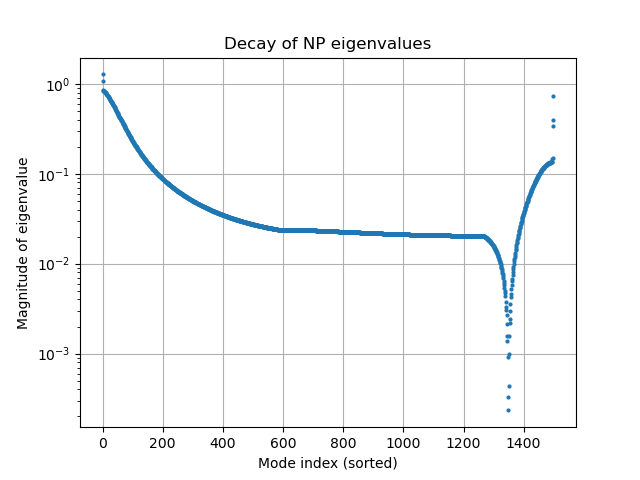}
    \includegraphics[width=0.5\linewidth]{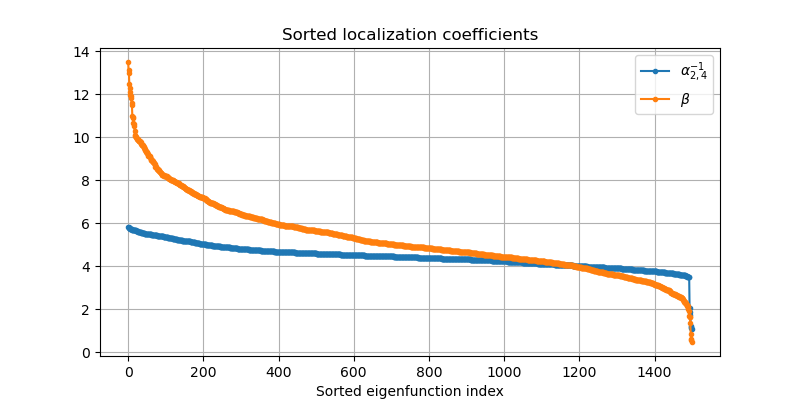}\\
        \includegraphics[width=0.49\linewidth]{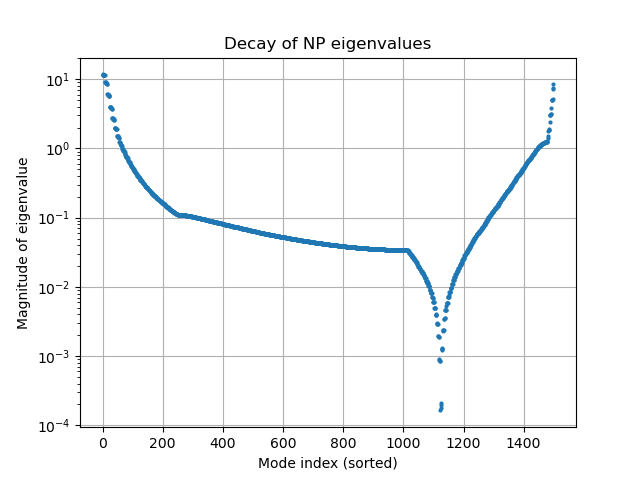}
    \includegraphics[width=0.5\linewidth]{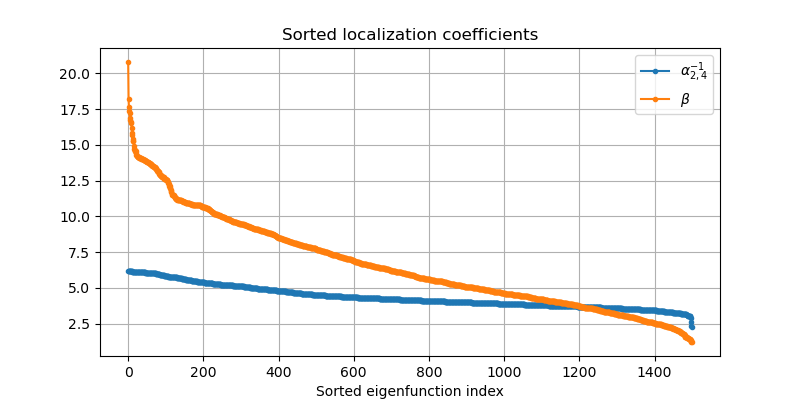}\\
    \includegraphics[width=0.49\linewidth]{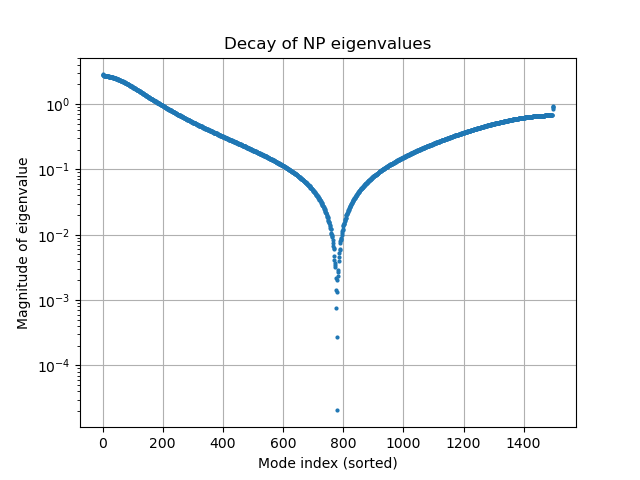}
    \includegraphics[width=0.5\linewidth]{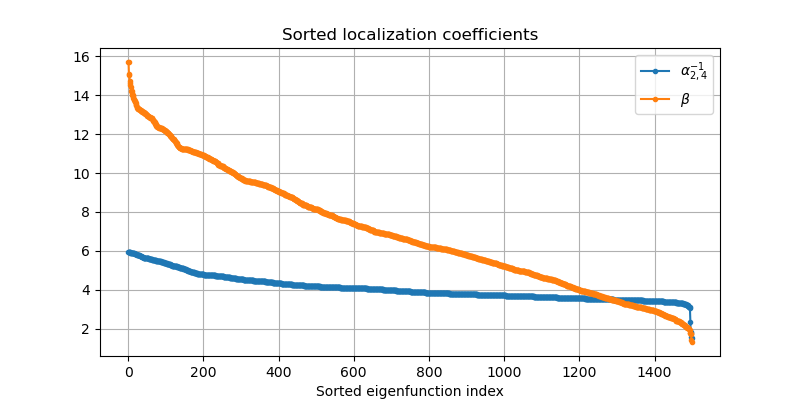}
    \caption{Eigenvalues and localization coefficients for curves 1, 2, and 3 from Figure~\ref{fig:curves}.}
    \label{fig:eigvals_and_coeffs}
\end{figure}
\begin{figure}
    \centering
    \includegraphics[width=0.49\linewidth]{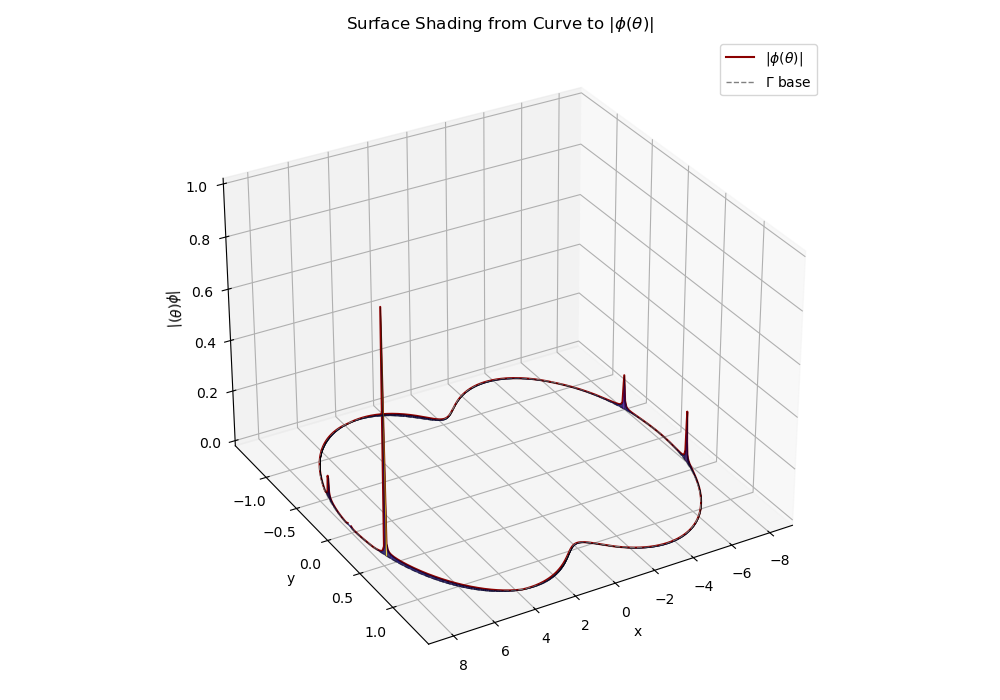}
        \includegraphics[width=0.49\linewidth]{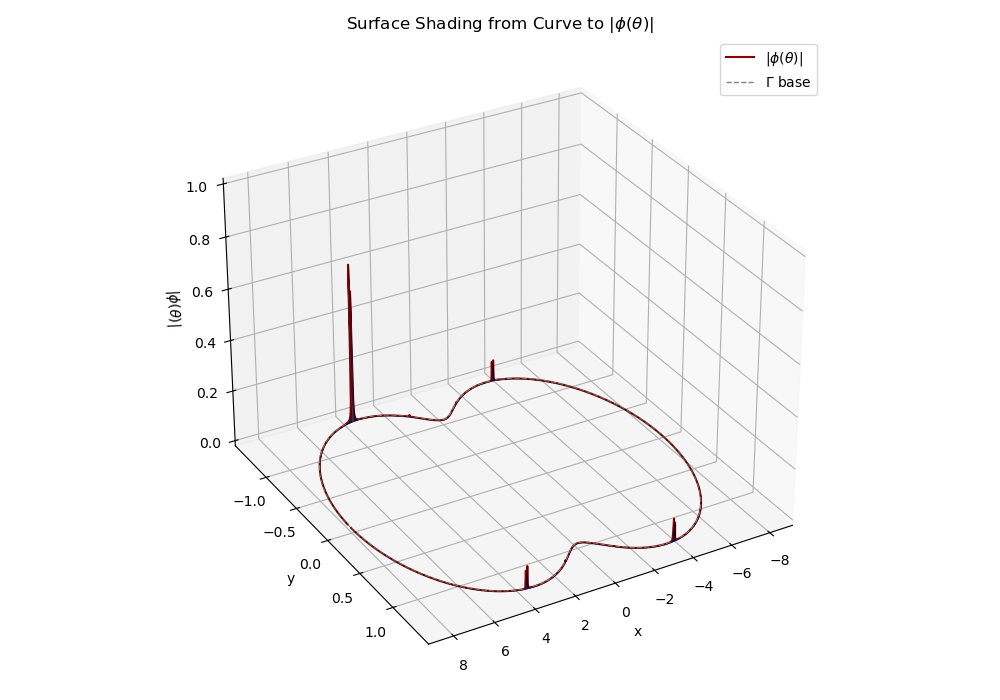}
    \includegraphics[width=0.49\linewidth]{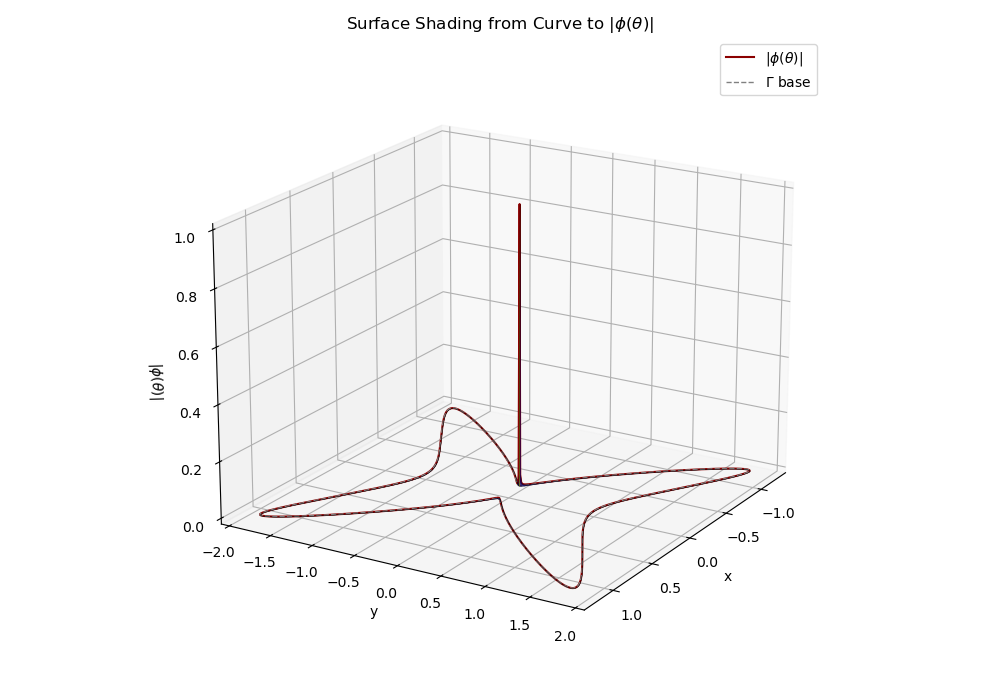}
        \includegraphics[width=0.49\linewidth]{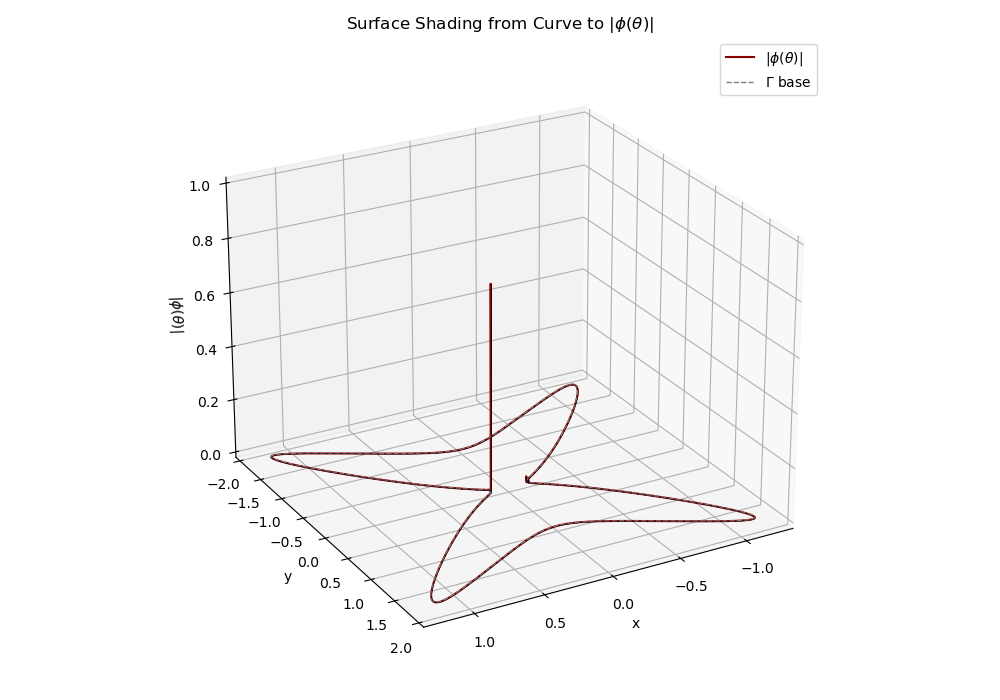}
         \includegraphics[width=0.49\linewidth]{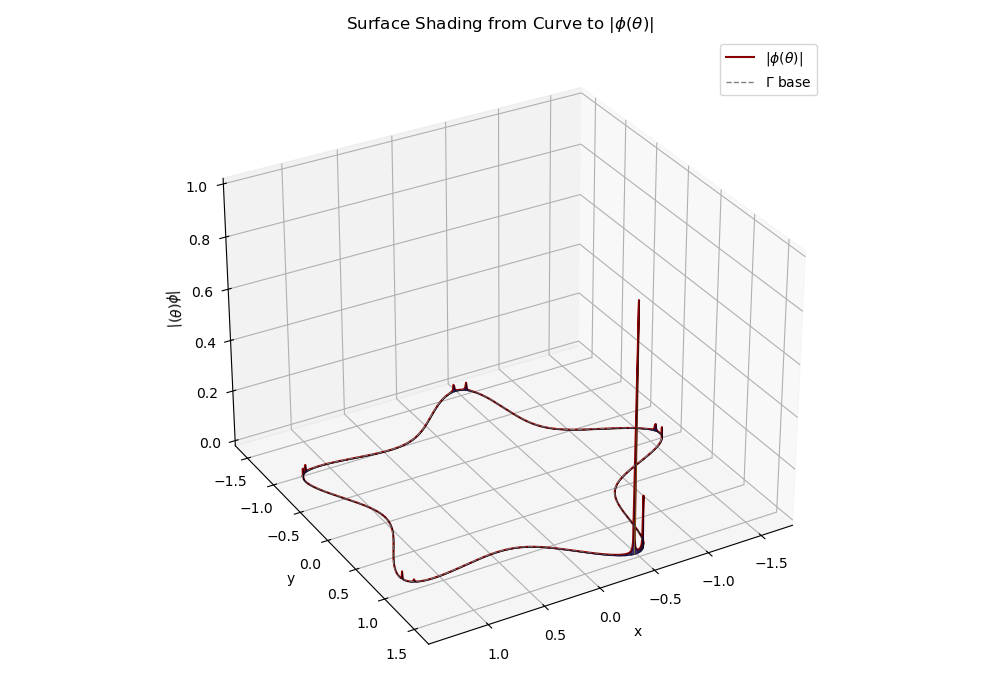}
        \includegraphics[width=0.49\linewidth]{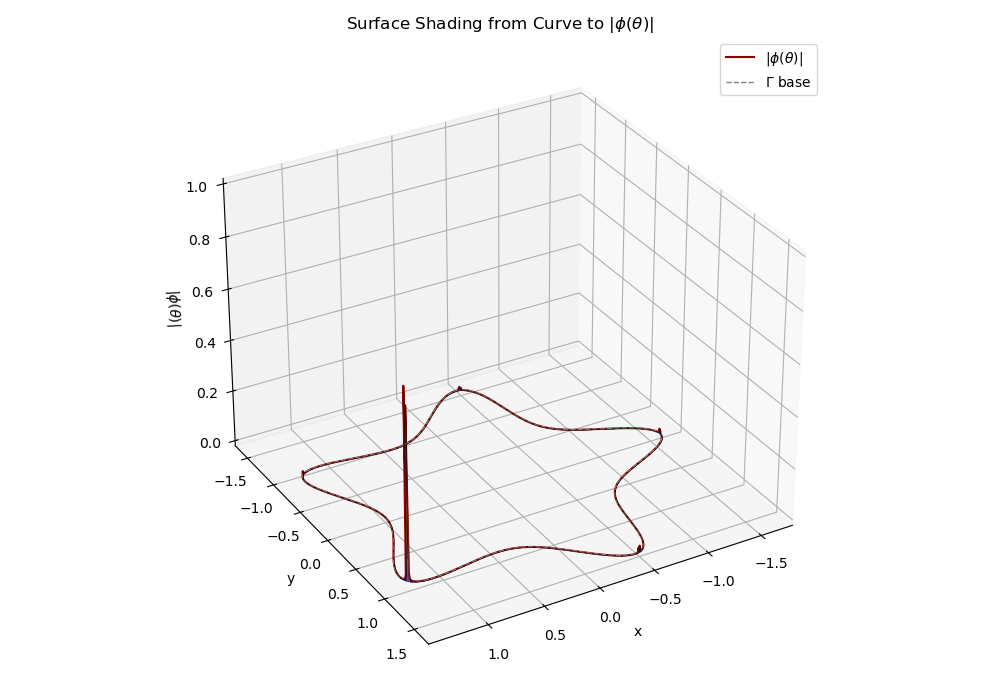}
    \caption{Absolute value distribution of the most localized eigenfunctions. Left column: according to $\alpha_{2,4}^{-1}$, right column: according to $\beta$.}
    \label{fig:3D_views}
\end{figure}

\subsection{Distance-to-boundary measure of localization in dimension two}\label{sec:dtb2D}
Figures~\ref{fig:2D_1} and~\ref{fig:recbum} show a two-dimensional domain $\Omega$ and a sequence of two-bump probability densities $u_d$ on $\Omega$ parametrized by the distance $d$ between the bump peaks. To account for the boundary effect, we choose to quantify the localization of $u_d$ in terms of the optimal transport cost $W_{\rm dist}(\mu_{u_d},|\Omega|^{-1}\lambda)$ with the transport cost function
\[
c(x,y)={\rm dist}(x,\partial\Omega) + {\rm dist}(y,\partial\Omega),\quad x,y\in\Omega.
\]
As shown in Figure~\ref{fig:W2cost2D}, this measure of localization has the desired strictly decaying behavior with increasing $d$ in spite of the immediate proximity of a bump to the boundary $\partial\Omega$ for large values of $d$.
\begin{figure}
    \centering
    \includegraphics[width=0.6\linewidth]{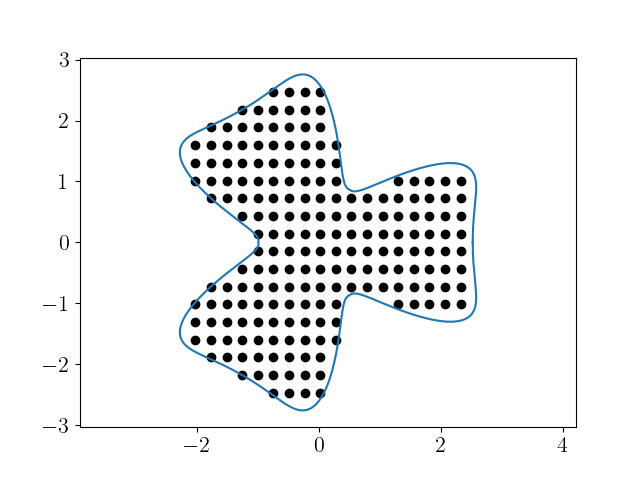}
    \caption{A two-dimensional domain and a discretizing point grid for the example of Section~\ref{sec:dtb2D}.}
    \label{fig:2D_1}
\end{figure}
\begin{figure}
    \centering
    \includegraphics[width=0.47\linewidth]{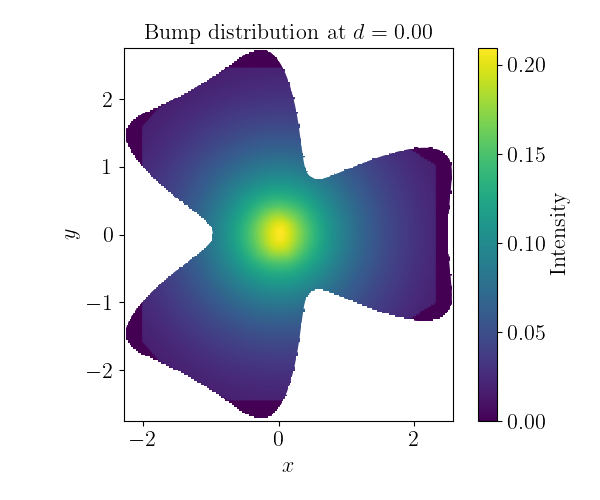}
    \includegraphics[width=0.47\linewidth]{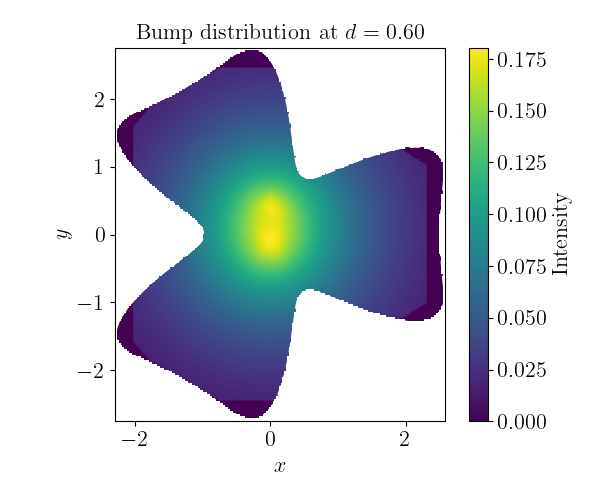}
    \includegraphics[width=0.47\linewidth]{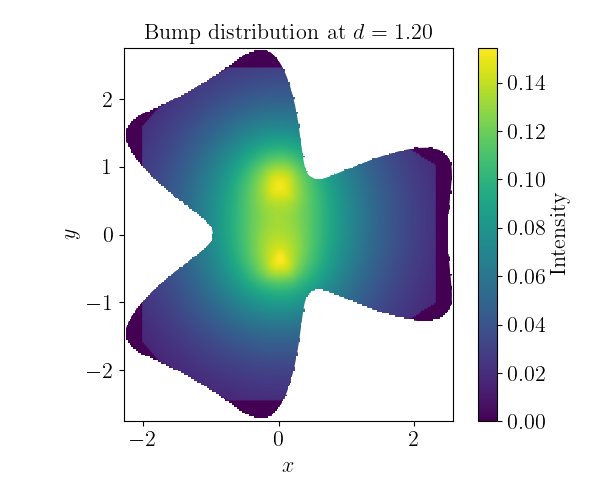}
    \includegraphics[width=0.47\linewidth]{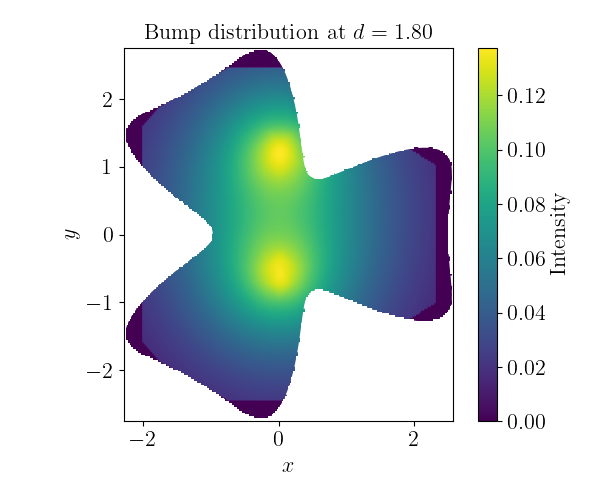}
    \includegraphics[width=0.47\linewidth]{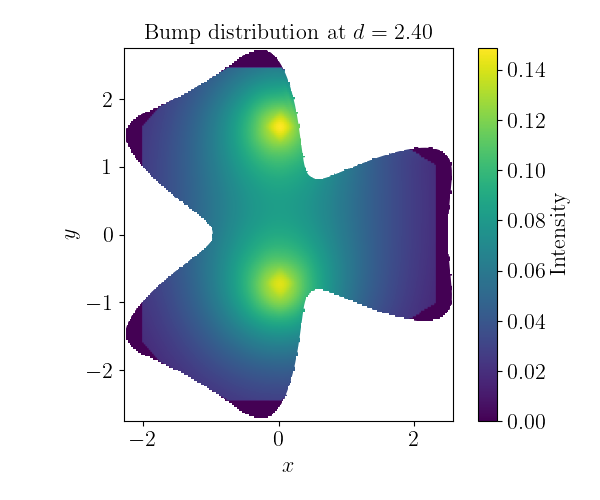}
    \includegraphics[width=0.47\linewidth]{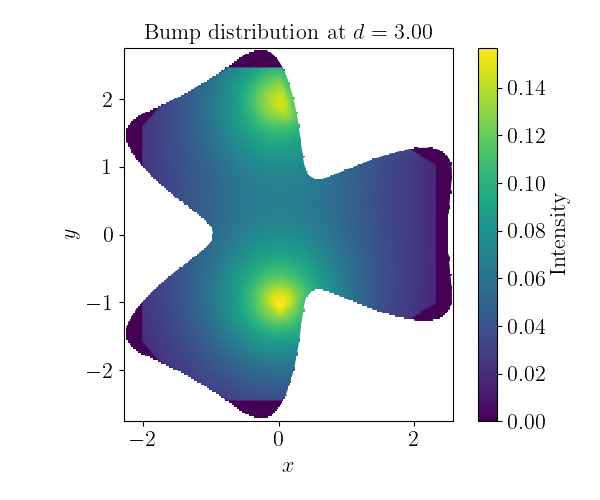}
    \caption{Receding bumps with distance $d$ for the example of Section~\ref{sec:dtb2D}.}
    \label{fig:recbum}
\end{figure}
\begin{figure}
    \centering
    \includegraphics[width=0.6\linewidth]{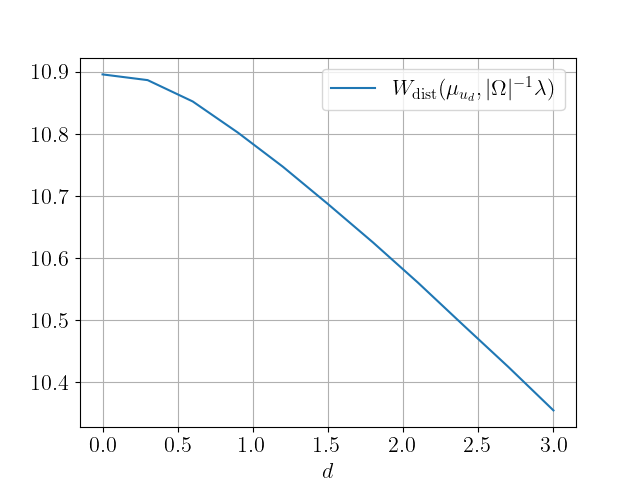}
    \caption{Distance-to-boundary measure of localization of densities from Figure~\ref{fig:recbum}.}
    \label{fig:W2cost2D}
\end{figure}

\section{Conclusion and outlook}\label{sec:conclusion}
After identifying limitations in existing approaches for quantifying the localization of functions, we introduced a new localization coefficient $\beta(\mu_u, \Omega)$ for measures $\mu_u$ with probability densities $u$ supported on $\Omega \subseteq \mathbf{R}^d$. Our method leverages optimal transport, specifically the Wasserstein-2 distance between $\mu_u$ and the normalized Lebesgue measure $|\Omega|^{-1}\lambda$ on $\Omega$. Numerical experiments demonstrate that $\beta(\mu_u, \Omega)$ outperforms traditional localization measures based on Lebesgue norms and mass concentration comparisons. Notably, we proved that in one dimension, $\beta(\mu_u, \Omega)$ coincides with the homogeneous Sobolev norm $\|u - |\Omega|^{-1}\|_{\dot{H}^{-1}(\Omega)}$, enabling more efficient computation in that setting. In addition, we observed and analyzed a boundary effect that can distort localization measurements via $\beta(\mu_u, \Omega)$, and we addressed this phenomenon explicitly in one and two spatial dimensions. A natural direction for future work is to improve estimates connecting the Wasserstein-2 distance and the homogeneous Sobolev norm in higher dimensions. We also anticipate that our new localization framework will facilitate advances in the design of PDE-governed and localization-sensitive physical phenomena, such as photonic nanojets~\cite{2023-phase-only_PNJ,2022-PNJ1}.

\section*{Acknowledgments}
This work was funded by the Villum Foundation research grant no. 58857. We thank Andreas Bacher Hannah of the Technical University of Denmark for his contribution to the numerical computation of eigenfunctions of the Neumann-Poincaré operator.

\bibliographystyle{abbrv}
\bibliography{references}

\begin{thebibliography}{10}

\bibitem{ammari2023quantum}
H.~Ammari, Y.~T. Chow, and H.~Liu.
\newblock Quantum ergodicity and localization of plasmon resonances.
\newblock {\em Journal of Functional Analysis}, 285(4):109976, 2023.

\bibitem{ammari2024anderson}
H.~Ammari, B.~Davies, and E.~O. Hiltunen.
\newblock Anderson localization in the subwavelength regime.
\newblock {\em Communications in Mathematical Physics}, 405(1):1, 2024.

\bibitem{brenier1991polar}
Y.~Brenier.
\newblock Polar factorization and monotone rearrangement of vector-valued
  functions.
\newblock {\em Communications on Pure and Applied Mathematics}, 44(4):375--417,
  1991.

\bibitem{burq2024delocalized}
N.~Burq and C.~Letrouit.
\newblock Delocalized eigenvectors of transitive graphs and beyond.
\newblock {\em arXiv preprint arXiv:2407.12384}, 2024.

\bibitem{Darafsheh-2021}
A.~Darafsheh.
\newblock Photonic nanojets and their applications.
\newblock {\em JPhys Photonics}, 3(2):022001, 2021.

\bibitem{figalli2021invitation}
A.~Figalli and F.~Glaudo.
\newblock {\em An invitation to optimal transport, Wasserstein distances, and
  gradient flows}.
\newblock EMS Press, 2021.

\bibitem{Filoche-2012-2}
M.~Filoche and S.~Mayboroda.
\newblock Universal mechanism for {A}nderson and weak localization.
\newblock {\em Proc. Natl. Acad. Sci.}, 109(37):14761--14766, 2012.

\bibitem{Filoche-2012}
M.~Filoche, S.~Mayboroda, and B.~Patterson.
\newblock Localization of eigenfunctions of a one-dimensional elliptic
  operator.
\newblock {\em Contemp. Math.}, 581:99--116, 2012.

\bibitem{Felix-2007}
S.~Félix, M.~Asch, M.~Filoche, and B.~Sapoval.
\newblock Localization and increased damping in irregular acoustic cavities.
\newblock {\em J. Sound Vib.}, 299:965--976, 2007.

\bibitem{Gielis-2003}
J.~Gielis.
\newblock A generic geometric transformation that unifies a wide range of
  natural and abstract shapes.
\newblock {\em Am. J. Bot.}, 90(3), 2003.

\bibitem{Grebenkov-2013}
D.~S. Grebenkov and B.-T. Nguyen.
\newblock Geometrical structure of laplacian eigenfunctions.
\newblock {\em SIAM Review}, 55(4), 2013.

\bibitem{hardy1929some}
G.~Hardy, J.~Littlewood, and G.~P\'olya.
\newblock Some simple inequalities satisfied by convex functions.
\newblock {\em Messenger Math.}, 58:145--152, 1929.

\bibitem{Heilman-2010}
S.~M. Heilman and R.~S. Strichartz.
\newblock Localized eigenfunctions: Here you see them, there you don’t.
\newblock {\em Not. Amer. Math. Soc.}, 57:624--629, 2010.

\bibitem{2023-phase-only_PNJ}
M.~Karamehmedovi\'c and J.~Gl\"uckstad.
\newblock Phase-only steerable photonic nanojets.
\newblock {\em Opt. Express}, 31(17):27255--27265, 2023.

\bibitem{2022-PNJ1}
M.~Karamehmedovi\'c, K.~Scheel, F.~L.-S. Pedersen, A.~Villegas, and P.-E.
  Hansen.
\newblock Steerable photonic jet for super-resolution microscopy.
\newblock {\em Opt. Express}, 30(23):41757--41773, 2022.

\bibitem{2024-localization_1D}
M.~Karamehmedovi{\'c} and F.~Triki.
\newblock Localization and the landscape function for regular
  {S}turm-{L}iouville operators.
\newblock {\em Commun. Math. Sci.}, 22(6):1733--1748, 2024.

\bibitem{kesavan2006symmetrization}
S.~Kesavan.
\newblock {\em Symmetrization and applications}, volume~3.
\newblock World Scientific, 2006.

\bibitem{Peyre-2018}
R.~Peyre.
\newblock Comparison between ${W}_2$ distance and $\dot {H}^{-1}$ norm, and
  {L}ocalization of {W}asserstein distance.
\newblock {\em ESAIM: COCV}, 24(4):1489--1501, 2018.

\bibitem{ramdas2017wasserstein}
A.~Ramdas, N.~Garc{\'\i}a~Trillos, and M.~Cuturi.
\newblock On wasserstein two-sample testing and related families of
  nonparametric tests.
\newblock {\em Entropy}, 19(2):47, 2017.

\bibitem{santambrogio2015optimal}
F.~Santambrogio.
\newblock Optimal transport for applied mathematicians.
\newblock {\em Birk{\"a}user, NY}, 55(58-63):94, 2015.

\bibitem{TaylorII}
M.~E. Taylor.
\newblock {\em Partial Differential Equations II: Qualitative Studies of Linear
  Equations}.
\newblock Springer, 2nd edition, 2011.

\bibitem{vandenBerg-2021-2}
M.~van~den Berg and T.~Kappeler.
\newblock Localization for the torsion function and the strong {H}ardy
  inequality.
\newblock {\em Mathematika}, 67:514--531, 2021.

\bibitem{vandenBerg-2021}
M.~van~den Berg, F.~D. Pietra, G.~di~Blasio, and N.~Gavitone.
\newblock Efficiency and localization for the first dirichlet eigenfunction.
\newblock {\em J. Spectral Theory}, 11:981--1003, 2021.

\bibitem{Vazquez-1982}
J.~L. V{\'a}zquez.
\newblock Sym{\'e}trisation pour {$u_t=\Delta\varphi(u)$} et applications.
\newblock {\em C. R. Acad. Sc. Paris}, 295:71--74, 1982.

\bibitem{Vazquez-2014}
J.~L. V{\'a}zquez and B.~Volzone.
\newblock Symmetrization for linear and nonlinear fractional parabolic
  equations of porous medium type.
\newblock {\em Journal de Math{\'e}matiques Pures et Appliqu{\'e}es},
  101(5):553--582, 2014.

\bibitem{villani2009optimal}
C.~Villani et~al.
\newblock {\em Optimal transport: old and new}, volume 338.
\newblock Springer, 2009.

\bibitem{Yamilov-2023}
A.~Yamilov, S.~E. Skipetrov, T.~W. Hughes, M.~Minkov, Z.~Yu, and H.~Cao.
\newblock Anderson localization of electromagnetic waves in three dimensions.
\newblock {\em Nature Phys.}, 19:1308--1313, 2023.

\end{thebibliography}

\end{document}